\documentclass[fleqn,final,3p,11pt]{elsarticle}
\usepackage{amsmath, amsthm, amscd, amsfonts, amssymb, graphicx, color}
\usepackage{amsthm,amssymb,amsmath,enumerate,mathtools}
\usepackage{hyperref}
\usepackage{cleveref}
\usepackage{breqn}
\usepackage{booktabs}   
\usepackage{xspace}
\usepackage{multirow}
\usepackage{pdfpages}
\usepackage{rotating}
\usepackage{grffile}
\usepackage{graphicx,booktabs,array}
\usepackage{tabularx}
\usepackage[export]{adjustbox}
\usepackage{rotating}
\usepackage{textcomp}
\usepackage{float}
\usepackage{subcaption}
\usepackage[toc,page]{appendix}
\usepackage[table,xcdraw]{xcolor}
\newtheorem{theorem}{Theorem}[section]
\newtheorem{lemma}[theorem]{Lemma}

\theoremstyle{definition}

\theoremstyle{remark}
\newtheorem{remark}[theorem]{Remark}
\numberwithin{equation}{section}
\def\im{\mathrm{i}}

\DeclareMathOperator{\erf}{erf}

\makeatletter
\def\@author#1{\g@addto@macro\elsauthors{\normalsize%
    \def\baselinestretch{1}%
    \upshape\authorsep#1\unskip\textsuperscript{%
      \ifx\@fnmark\@empty\else\unskip\sep\@fnmark\let\sep=,\fi
      \ifx\@corref\@empty\else\unskip\sep\@corref\let\sep=,\fi
      }%
    \def\authorsep{\unskip,\space}%
    \global\let\@fnmark\@empty
    \global\let\@corref\@empty  
    \global\let\sep\@empty}%
    \@eadauthor={#1}
}
\makeatother

\input{mathrsfs.sty}
\begin{document}
\begin{frontmatter}

\title{A Product Integration Method for the Approximation of the Early Exercise Boundary in the American Option Pricing Problem\tnoteref{mytitlenote}}

\author{Khadijeh Nedaiasl\corref{cor1}\fnref{fn1}}
\ead{nedaiasl@iasbs.ac.ir}
\cortext[cor1]{Corresponding author}

\author{Ali Foroush Bastani\fnref{fn2}}
\ead{bastani@iasbs.ac.ir}

\author{Aysan Rafiee\fnref{fn3}}
\ead{Aysan.Rafiee@outlook.com}
\address{Institute for Advanced Studies in Basic Sciences, Zanjan, Iran.}


\begin{abstract}
In this paper, an integral equation representation for the early exercise boundary of an American option contract is considered. Thus far, a number of different techniques have been proposed in the literature to obtain a variety of integral equation forms for the early exercise boundary, all starting from the Black-Scholes partial differential equation. We first present a coherent categorization of exiting integral equation methodologies in the American option pricing literature. In the reminder and based on the fact that the early exercise boundary satisfies a fully nonlinear weakly singular non-standard Volterra integral equation, we propose a product integration approach based on linear barycentric rational interpolation to solve the problem. The price of the option will then be computed using the obtained approximation of the early exercise boundary and a barycentric rational quadrature. The convergence of the approximation scheme will also be analyzed. Finally, some numerical experiments based on the introduced method are presented and compared to some exiting approaches.
\end{abstract}


\end{frontmatter}
\section{Introduction}\label{s1}\setcounter{equation}{0}
Volterra integral equations (VIEs) are of fundamental importance in the mathematical modelling of many scientific, economic, physical, chemical and biological phenomena \cite{corduneanu1991integral, jerri1999introduction}. Taking into account the fact that a general initial value problem could be rewritten as a Volterra integral equation and also due to the basic role of VIEs in the study of evolutionary processes, Volterra equations have gained much popularity in the functional and numerical analysis fields and many theoretical and numerical efforts have been devoted to study their solutions and properties (see e.g. \cite{brunner, hack}).

In recent years, integral equation models have also found their way into the Wall Street and some practical financial problems, mainly from the field of financial option pricing and hedging are now reformulated as Volterra integral equations (see e.g. \cite{chen, evans,  patrik2, keller, shev2} and the many references therein).

This line of research started with the pioneering contributions of Kim \cite{kim}, Jacka \cite{jacka} and Carr et al. \cite{carr1} who derived nonlinear integral representations for the ``early exercise premium'' where the underlying asset follows a geometric Brownian motion. Soon after, a bunch of numerical methods for American option pricing using these integral representations were proposed by Broadie and Detemple \cite{broadie1996american}, Huang et al. \cite{huang1996pricing}, Ju \cite{ju}, AitSahlia and Lai \cite{lai2}, and Kallast and Kivinukk \cite{kallast} among others.

Based on the fact that integral operators have a smoothing character and could potentially increase the regularity properties of their input functions, the methodology of transforming partial differential equations into equivalent integral equations, known in the literature as ``Boundary Integral Equation Method'' has been a widely developed field within the scientific computing community \cite{stas}. In the context of Black-Scholes partial differential equation (PDE) considered as a parabolic free boundary problem, such an approach has been employed successfully based on different transformation techniques (e.g. Fourier, Laplace, Mellin, etc.) to arrive at a variety of integral equation formulations of the problem \cite{MR2087015}.

Although some studies in the finance literature have criticized the use of integral equation methods in option pricing\footnote{Due mainly to their low speed and high computational costs.}, in recent years this point of view has changed and recent research has shown a promising speed-accuracy performance for the integral equation approach \cite{andersen}. This has led some researchers to put forth their efforts to explore and extend these integral representations with the hope to make them a method of choice in real-time computing frameworks.

It is worthwhile to mention that the widespread appearance of IEs in finance will potentially open new avenues in the study of some integral equation families which have been previously studied only in some restricted senses (e.g. non-standard Volterra integral equations \cite{brunner, guan}). Moreover, there is also incentives to invent new tools and techniques in this rapidly developing field of study to accommodate for the arising problems and challenges.

Due to the fact that these integral equation representations are usually derived from the Black-Scholes partial differential equation, starting from different departure points by employing a wide range of transforms and resulting in a variety of forms with different characteristics, it will be helpful to have a comprehensive categorization and a coherent presentation of these various forms in order to gain some insight into their behaviors. This task will also be of help when we try to extend these techniques to other asset price dynamics and also option payoff structures.

Recently, Chiarella and his coworkers \cite{chiarella2014numerical, chia} have provided a survey on integral representations of the optimal exercise boundary, arising from the American option pricing problem. As a first contribution of this kind, their work could be extended to include more recent developments in the field, as well as some less well-known representations in a unified manner. In this respect, the first part of this paper is concerned with a comprehensive review of the existing approaches in the literature for driving the integral equation representations of the early exercise boundary. We also present some general considerations concerning the existence and uniqueness issue for these integral equations.

Among the existing integral equation reformulations of the early exercise boundary, Kim's representation \cite{kim} is of particular interest, partly due to the financial interpretation of each term in the equation. This has resulted in the development of some approximation techniques in the finance literature to solve this equation \cite{lai2, ju, kallast}. Much of the numerical research in this area is based on direct discretization of the integral terms, called in the literature of integral equations as the Nystr\"{o}m \cite{atkinson} or quadrature method \cite{hack}. However, there is still much room for improving the performance of numerical approaches based on interpolatory quadrature rules to solve the problem at hand.

Taking into account the fact that the early exercise boundary has some kind of singularity near the expiry (see e.g. \cite{evans, stamicar}) and noting that this knowledge must be incorporated in the design of the numerical scheme, we consider here a one-dimensional reformulation of Kim's integral equation proposed by Hou, et al. \cite{little} and employ a generalization of the Nystr\"{o}m method, called the product integration method \cite{atkinson}, specifically designed to tackle this singular behavior. More precisely, we employ an approximation of the kernel based on linear barycentric rational interpolation to manage the weakly singular nature of the integral equation \cite{cuminato, orsi, hoog}.

In this respect, after a brief review of the existing numerical approaches utilized for the approximation of  the early exercise boundary, we provide  theoretical and numerical evidence that the product integration method based on linear barycentric rational interpolation is an efficient way to approximate the solution. In the sequel, the integral representation of the American option price and its numerical approximation will be considered and an upper bound for the incurred error will be given.

The structure of the paper is as follows. After presenting a survey of existing techniques to arrive at integral equation representations for the early exercise boundary in Section \ref{option}, we introduce the product integration method based on barycentric rational interpolation to approximate the free boundary as well as a convergence analysis of the numerical method in Section \ref{NMEEB}. We then employ the corresponding barycentric rational quadrature to find the price of the option from its integral representation in Section \ref{thistable}. We have performed some numerical experiments in Section \ref{NE} to confirm the theoretical findings of the paper and also a detailed comparison is made between the presented method and some competing approaches. Section \ref{conclusion} concludes the paper by pointing out to some research questions worthy of consideration in the future.

\section{From Option Valuation to Integral Equations}\label{option}
In this and the following sections, we assume that the asset price process, $\{S(t),t\geq 0\}$, follows a lognormal diffusion of the form
 \[dS(t) = (r-\delta)S(t)dt + \sigma S(t)dW(t),\]
in which $\{W(t),t\geq 0\}$ is the standard Wiener process, $r$ is the constant interest rate, $\sigma$ is the constant volatility and $\delta$ is the continuous proportional dividend yield.

Our aim here is to give a brief overview of different methods to derive integral equations describing the early exercise boundary of an American call or put option. For this purpose, we start from the famous Black-Scholes PDE of the form
\begin{gather}\label{pde}
\frac{\partial V}{\partial t}+\frac{1}{2}\sigma^{2}S^{2}\frac{\partial^{2}V}{\partial S^{2}} +(r-\delta)S\frac{\partial V}{\partial S}-rV = 0,
\end{gather}
in which $V(t, S)$ describes the price of an option at time $t$, when the underlying security price is equal to $S = S(t)$.

The associated boundary and initial conditions in the case of an American put option with $V(t, S)\equiv P(t, S)$ are of the form:
\begin{align}\label{callcon}
P(t, S)&=K - S, \quad \text{for} \quad S = \mathcal{B}(t), \quad 0\leq t <T, \\\label{smoothcal}
\frac{\partial P}{\partial S}(t, S)&=-1,  \quad \text{for} \quad S = \mathcal{B}(t), \quad 0\leq t <T,\\
P(T, S)& = \max\{0, K-S \}, \quad
\lim_{S \rightarrow \infty} P(t, S) = 0,
\end{align}
and the corresponding conditions for an American call with $V(t, S)\equiv C(t, S)$ could be written as:
\begin{align}\label{putcon}
C(t, S)&=S - K, \quad \text{on} \quad S = \mathcal{B}(t), \quad 0\leq t <T, \\\label{smoothput1}
\frac{\partial C}{\partial S}(t, S)&=1,  \quad \text{on} \quad S = \mathcal{B}(t), \quad 0\leq t <T,\\
C(T, S)& = \max\{0, S - K \}, \quad \lim_{S \rightarrow 0} C(t, S) = 0.\label{tah}
\end{align}
In the above expressions, $K$ is the exercise price of the option, $T$ is the expiry and $\mathcal{B}(t)$ is a free boundary corresponding to the ``optimal exercise price" or the ``early exercise boundary", to be determined alongside the option price\footnote{As a time-dependent function, $\mathcal{B}(t)$ could be utilized for dividing the hold and exercise regions of the option.}.

In recent years, there have been many efforts to find these unknowns by different analytical and numerical approaches. Among the semi-analytical techniques, one could mention the quadratic approximation method of  Barone-Adesi and Whaley \cite{barron}, two-point and three-point maximum methods of Bunch and Johnson \cite{bunch1992simple} and the lower and upper bound approximation methods of Broadie and Detemple \cite{broadie1996american}. From a numerical discretization point of view,  the finite difference \cite{duffy}, finite element \cite{achdou} and  spectral methods \cite{chen2012new} could also be mentioned.

As an alternative and to obtain an expression for the solution of PDEs, we could apply a wide range of transform techniques available in the literature \cite{stamicar, shev2} to reduce the problem dimension. Roughly speaking, transform methods convert the PDE into one or more ordinary differential equations which by solving them and applying the inverse transform on the solutions we obtain an expression for the price. The next natural step is to use the smooth pasting condition (\ref{smoothcal}) or (\ref{smoothput1}) to arrive at a nonlinear integral equation for the early exercise boundary.

Among other approaches to represent the solution of the pricing equation, we could also mention the Green's function method \cite{evans} and optimal stopping representation \cite{peskir}. In the following, we give a brief outline of these approaches towards tackling the pricing problem:

\begin{description}
\item [(Complete and Incomplete) Fourier Transform Approach] \label{A}
Applying the Fourier transform on equation (\ref{pde}) leads to a nonlinear integral equation  for the free boundary, $\mathcal{B}(t)$, defined recursively and described in detail for the zero divided case in \cite{stamicar} and also for  the non-zero dividend case in  \cite{shev2}. In both cases, the obtained integral equations are of non-standard Volterra type (see the Appendix \ref{forieh} for more details).

\item [Laplace Transform Approach]
Utilizing the Laplace transform on equation (\ref{pde}) will result in an integral equation for the location of the free boundary, $ \mathcal{B}(t) $ \cite{gada1}. In this case, the  nonlinear  integral equation is of the Fredholm type with an unbounded domain of integration (for more details see the Appendix \ref{laplas}).

 \item [Mellin Transform Approach]
Using the Mellin  transform technique and employing the convolution property of it (see e.g. \cite{patrik2}), we obtain a class  of nonlinear Volterra integral equations of the second kind \cite{brunner}. As it is shown in \cite{panini}, this kind of Volterra integral equation is equivalent to the one obtained from the optimal stopping approach (see the Appendix \ref{melina}).

\item [Green's Function Approach]
    Some researchers in the field have adopted the method of Green's functions or fundamental solutions \cite{stack} for solving Eq. (\ref{pde})
      which will result in a family of integral and integro-differential equations of Volterra type \cite{chen, evans, keller} (see the Appendix \ref{green1}).
 \item[Optimal Stopping Approach]  Employing the risk-neutral valuation approach of Cox and Ross \cite{cox} and Kim \cite{kim} obtained an integral equation for the early exercise boundary of an American option
as the continuous limit of the valuation formula that allow early exercise at a finite number of points in time
   (see the Appendix \ref{stop}). He also obtained an integral representation for the value of the option based on the critical stock price. It should be noticed that Jamshidian in \cite{jam} has obtained the same representation as Kim \cite{chia, jam} via the Duhamel principle. Furthermore, for the general discrete dividend case, an integral equation for the early exercise boundary of American options is studied extensively in \cite{vellek1, vellek2}.
\end{description}
In Tables \ref{table:1} and \ref{table:2}, we have outlined all of the above forms and also the integral equation classes (\ref{kimnondiv}) and (\ref{kimdiv}) which will be introduced in the sequel. The above approaches provide a variety of integral equations each with specific flavors. Among them, we only mention the following:
\begin{itemize}
\item[$\bullet$] Weakly singular IEs (see Eq. (\ref{aa1})),
\item[$\bullet$] Recursive nonlinear IEs (see Eq. (\ref{Aa2})),
\item[$\bullet$] Urysohn IEs of the first kind (see Eq. (\ref{bb1})),
\item[$\bullet$] Delayed Volterra IEs (see Eq. (\ref{2.4})),
\item[$\bullet$] Fully nonlinear weakly singular Volterra IEs (see Eq. (\ref{shevon})).
\end{itemize}
As a natural question, one could ask whether and how these integral equations are interrelated? Although this question in unanswered in the general case, the relation between Kim's representation and the one obtained from the Mellin transform approach presented by (\ref{melin}) have been studied in \cite{patrik2}.  Also, it is worth mentioning that Kim's representation for the price could also be obtained using the Fourier transform (see \cite{underwood2002integral} for more details).
Recently, Alobaidi \textit{et al.} have shown that the integral equations obtained from Mellin and Laplace transform are equivalent to the one derived from Green's function approach \cite{alobaidi2014integral}.
\begin{table}[ht!]
\centering
\begin{tabular}{ |p{3.2cm}|p{8.3cm}|p{2cm}| }
 \multicolumn{3}{c}{} \\
 \hline
\center{Approach} & \center{Equation} &\begin{center} \small{IE Kind}\end{center} \\
 \hline
 \begin{center}Fourier Transform \cite{chia, chiarella2014numerical, stamicar, shev2}\end{center}& \begin{center}$ u(t)=g(t,u(t) )+\int_{0}^{t}k(t,s,u(t),u(s))\mathrm{d}s $\end{center}  & \begin{center} \small{Nonlinear Weakly Singular Volterra} \end{center}\\
\hline
\begin{center}Laplace Transform  \cite{knessl, gada1}\end{center}& \begin{center}$ g(t)=\int_{a}^{b}k(t,s,u(s))\mathrm{d}s $\end{center}   \begin{center}$ g(t)=\int_{0}^{\infty}k(t,s,u(s))\mathrm{d}s $\end{center}
& \begin{center} \small{ First kind nonlinear Fredolm \& Weakly Singular Fredholm} \end{center} \\
\hline
\begin{center}Mellin Transform \cite{patrik2, panini}\end{center} &  \begin{center}$ u(t)=g(t,u(t) )+\int_{t}^{b}k(t,s,u(t),u(s))\mathrm{d}s $\end{center} & \begin{center} \small{Nonlinear Weakly Singular Volterra} \end{center} \\
\hline
\begin{center}Green's Function \cite{chen,  chen1, chia, chiarella2014numerical, evans,  kim, keller}\end{center}& \begin{center}$ u(t)=g(t,u(t) )+\int_{t}^{b}k(t,s,u(t),u(s),u'(s))\mathrm{d}s $\end{center}
 \begin{center} $ u(t)=g(t,u(t) )
+\int_{0}^{t}k(t,s,u(t),u(s))\mathrm{d}s$ \end{center}  & \begin{center} \small{Nonlinear Weakly Singular Volterra Integral and Integro-differential} \end{center} \\
\hline
\begin{center}Optimal Stopping \cite{kim, peskir}\end{center}&
 \begin{center} $ u(t)=g(t,u(t) )
+\int_{0}^{t}k(t,s,u(t),u(s))\mathrm{d}s$ \end{center}  & \begin{center}  \small{Nonlinear Weakly Singular Volterra} \end{center} \\
\hline
\end{tabular}
\caption{Integral equation types arising from the American option pricing problem.}
\label{table:1}
\end{table}
\begin{table}[ht!]
\centering
\begin{tabular}{ |p{3.2cm}|p{8.3cm}|p{2cm}| }
 \multicolumn{3}{c}{} \\
 \hline
\center{Approach} & \center{Equation} & \begin{center}\small{IE Kind}\end{center} \\
\hline
\begin{center}Hou et.al 's  \cite{little}\end{center} & \begin{center} $ g(t,u(t))=\int_{0}^{t}k(t,s,u(t),u(t-s))\mathrm{d}s $
 \end{center} & \begin{center} \small{Nonlinear Weakly Singular Volterra }\end{center}\\
\hline
\begin{center}Kim's (2013) \cite{ kim2}\end{center} & \begin{center} $ u(t)=g(t,u(t) )
+\int_{0}^{t}\frac{1}{\sqrt{t-s}}k(t,s,u(t),u(s))\mathrm{d}s $ \end{center}
&  \begin{center}\small{Nonlinear Weakly Singular Volterra}\end{center}\\
\hline
\end{tabular}
\caption{One Dimensional Integral Equations}
\label{table:2}
\end{table}
\subsection{Kim's Integral Equation Representation} Among the above mentioned ways to arrive at integral equation representations, Kim's approach belonging to the optimal stopping category is an elegant way to characterize the behavior of the early exercise boundary in the American option pricing literature \cite{chia, chiarella2014numerical, kim}. In this approach, it could be shown (see e.g. \cite{chen1,  kim, kim2}) that the early exercise boundary, $\mathcal{B}(t)$ of an American put option satisfies a weakly singular Volterra integral equation of the form
\begin{align}\label{kim}
K-\mathcal{B}(t)=p^{E}(t,\mathcal{B}(t)) +& \int_{0}^{t}[rKe^{-r(t-s)}\aleph(-d_{2}(\mathcal{B}(t),t-s,\mathcal{B}(s)))\\\notag
-&\delta \mathcal{B}(t) e^{-\delta (t-s)}\aleph(-d_{1}(\mathcal{B}(t),t-s,\mathcal{B}(s)))]\mathrm{d}s,
\end{align}
in which $ p^{E}(t, S) $ represents the price of an otherwise equivalent European counterpart given by
\begin{equation}\label{IEs}
p^{E}(t, S)= Ke^{-rt}\aleph(- d_{2}(S,t,K))-\mathcal{B}(t)e^{-\delta t}\aleph(-d_{1}(S,t,K)),
\end{equation}
and $\aleph(.) $ is the standard cumulative normal distribution function. Furthermore, the functions $ d_{1}(x,t,y) $ and $ d_{2}(x,t,y) $ are defined respectively by
\begin{equation}\label{d1d2}
d_{1}(x,t,y) =\frac{\log(\frac{x}{y})+ (r-\delta + \frac{\sigma^{2}}{2})t}{\sigma\sqrt{t}}, \quad d_{2}(x,t,y) = d_{1}(x,t,y) - \sigma \sqrt{t}.
\end{equation}
In this case, the price of the American option, represented by $P(t,S)$, could be recovered from $\mathcal{B}(t)$ by the expression
\begin{align}\label{price}
P(t,S)=p^{E}(t,S) +& \int_{0}^{t}rK e^{-r(t-\xi)}\aleph (-d_{2}(S, t-\xi, \mathcal{B}(\xi)))\mathrm{d}\xi\\
   - &\int_{0}^{t} \delta S e^{-\delta (t- \xi)} \aleph (-d_{1}(S, t- \xi, \mathcal{B}(\xi)))  \mathrm{d}\xi,\notag
\end{align}
which is known in the literature as the ``early exercise premium representation'' (see \cite{kim} for more details).

Due to the appearance of the cumulative normal distribution term, $\aleph(.)$ in (\ref{kim}) and noting that it has an integral representation, we are faced with a two-dimensional integral equation. This will make the problem hard from a numerical point of view. Some researchers have tried to reduce this two-dimensional equation into a one-dimensional expression to improve the numerical and analytic tractability of the  integral representation which is the subject of the following subsection.
\subsection{Converting  Kim's Representation into a One-Dimensional Form}
Hou et al., in \cite{little} have proposed a technique to reduce (\ref{kim}) to a one dimensional integral equation. Their method is based on
 replacing the term $\mathcal{B}(t)$ in (\ref{kim}) by $\epsilon\mathcal{B}(t)$, differentiating w.r.t. $\epsilon$ and taking the limit as $\epsilon$ tends to zero. In this way, one obtains the equation
\begin{align}\label{2.4}
\mathcal{B}(t) \Big\{ &\sigma e^{- \delta t -\frac{1}{2}d_{1}(\mathcal{B}(t), t, K)^{2}}+ \delta \sqrt{2\pi t }\Big\}
 = K r \sqrt{2 \pi t}\notag \\
 &+ \delta \mathcal{B}(t)\sqrt{t}\int_{0}^{t}e^{-\delta s-\frac{1}{2}d_{1}(\mathcal{B}(t), s, \mathcal{B}(t-s))^{2}}\Big( \frac{d_{2}(\mathcal{B}(t), s, \mathcal{B}(t-s))^{2}}{s} \Big)\mathrm{d}s\notag\\
&-Kr\sqrt{t}\int_{0}^{t}e^{-r s-\frac{1}{2}d_{2}( \mathcal{B}(t), s, \mathcal{B}(t-s))^{2}}\Big( \frac{d_{1}( \mathcal{B}(t), s, \mathcal{B}(t-s))^{2}}{s} \Big)\mathrm{d}s,
\end{align}
which has the general form represented in the first row of Table \ref{table:2}. It should also be noted that the numerical solution of equation (\ref{2.4}) is considered in \cite{little}.

Using similar ideas, Kim \textit{ et al.} \cite{kim2}  obtain an integral equation in the zero-divided case of the form
\begin{align}\label{kimnondiv}
\mathcal{B}(t) \aleph(d_{1}&(\mathcal{B}(t),t, K))+\mathcal{B}(t) \frac{1}{\sigma\sqrt{2\pi t}}K\exp\Big( -\frac{1}{2}d_{1}(\mathcal{B}(t),t, K)^{2}\Big) \notag\\
&= \frac{1}{\sigma \sqrt{2\pi t}}K\exp\Big( -\Big[rt + \frac{1}{2} d_{2}(\mathcal{B}(t),t, K)^{2}\Big]\Big)\notag\\
&+ rK\int_{0}^{t}\frac{1}{\sigma\sqrt{2\pi (t-\xi)}} \exp\Big(-\Big(r(t-\xi)+ \frac{1}{2}d_{2}(\mathcal{B}(t),t-\xi, \mathcal{B}(\xi) )^{2}\Big)\Big)\mathrm{d}\xi,
\end{align}
and the nonlinear integral equation
\begin{align}\label{kimdiv}
-\mathcal{B}(t)\exp(-\delta t) \aleph &\Big(d_{1}(\mathcal{B}(t),t, K)\Big) +\frac{K}{\sigma \sqrt{2\pi t}} \exp \Big(- \Big(rt + \frac{1}{2}d_{2}(\mathcal{B}(t),t, K)^{2}\Big)  \Big)\notag\\
& -\frac{\mathcal{B}(t)}{\sigma \sqrt{2\pi t}} \exp \Big(-\Big (\delta t + \frac{1}{2}d_{1}(\mathcal{B}(t),t, K)^{2}\Big)  \Big)\notag\\
&+\int_{0}^{t}\frac{1}{\sigma \sqrt{2\pi (t-\xi)}}\Big[ rK \exp \Big(- r(t-\xi) - \frac{1}{2} d_{2}(\mathcal{B}(t),t-\xi, \mathcal{B}(\xi))^{2}\Big)\notag\\
&-\delta \mathcal{B}(t) \exp \Big(- \delta(t-\xi) - \frac{1}{2} d_{1}(\mathcal{B}(t),t-\xi, \mathcal{B}(\xi))^{2}\Big)\Big] \mathrm{d}\xi \notag \\
&- \delta \int_{0}^{t} \mathcal{B}(t) \exp(-\delta (t - \xi))  \aleph \Big(d_{1} \Big(\mathcal{B}(t), t-\xi, \mathcal{B}\xi)\Big)\Big)\mathrm{d}\xi = 0,
\end{align}
in the dividend paying case.
\begin{remark}
Consider the integral equation
\begin{equation}\label{asli}
u(t)=g(t,u(t))+\int_{0}^{t}\frac{1}{(t-s)^{\alpha}}k_{1}(t,s,u(t),u(s))\mathrm{d}s+\int_{0}^{t}k_{2}(t,s,u(t),u(s))\mathrm{d}s,
\end{equation}
with $0\leq\alpha <1$, where the forcing function $ g $ and the kernels $k_{1}$ and  $k_{2}$ are given and $ u(t) $ is an unknown function to be determined. It is easily seen that equations (\ref{kim}), (\ref{kimnondiv}) and (\ref{kimdiv}) all are of this general form for suitable
$k_{1},k_{2}$ and $\alpha$.
\end{remark}
\subsection{Existence and Uniqueness Issue}
In recent years, a number of researchers have dealt with the existence and uniqueness issue for the free boundary problem resulting from the American option and its early exercise boundary, based primarily on fixed point theorems and also a probabilistic approach \cite{chen,jacka,peskir,myneni} .

Chen and Chadam \cite{chen} prove the existence and uniqueness for the pricing problem in free boundary form (\ref{pde})-(\ref{callcon}) via the Schauder fixed point theorem and also some comparison theorems. They prove the existence and uniqueness of the pair $(P, \mathcal{B})$ as well as the continuity and monotonicity of $\mathcal{B}(t)$. On the other hand, Jacka \cite{jacka} using a probabilistic approach has proved that the early exercise boundary is unique under a condition which will cause some difficulties in the numerical calculation procedure. Myneni \cite{myneni} stated in his paper that ``the uniqueness and regularity of the stopping boundary from this integral equation remain open''. Peskir \cite{peskir} employed a change-of-variables formula with local time on curves to prove, in a nine-step process, the uniqueness of the solution for the equation
\begin{eqnarray*}\label{peskireq}
K-\mathcal{B}(t)=&e^{-r(T-t)}\int_{0}^{K}\aleph\Big(\dfrac{1}{\sigma
\sqrt{T-t}}\Big(\log\Big(\dfrac{K-s}{\mathcal{B}(t)}\Big)-(r-\frac{\sigma^{2}}{2})(T-t)\Big)\Big)\mathrm{d}s\\\notag
&+rK\int_{0}^{T-t}e^{-rs}\aleph\Big(\dfrac{1}{\sigma \sqrt{s}}\Big(\log\Big(\dfrac{\mathcal{B}(t+s)}{\mathcal{B}(t)}\Big)-(r-\dfrac{\sigma^{2}}{2})s\Big)\Big)\mathrm{d}s,
\end{eqnarray*}
which is a Volterra nonlinear integral equation describing the early exercise boundary.

It must be stressed here that research on the existence and uniqueness theorems for the early exercise boundary from the integral equations point of view is an ongoing issue which is of independent interest in the field. In fact, by imposing more restrictive conditions on the forcing function and the kernel, one obtains the required result using classical fixed point theorems (for more details on the case $\alpha =0$ see e.g. \cite{nedaiasl2017numerical}) but proving a theorem with minimal conditions compatible with the structure of the integral equations will require the extension of some advanced techniques in the theory of integral equations.
\section{Numerical Methods for the Early Exercise Boundary}\label{NMEEB}
Based on the fact that a closed form analytical solution for Eq. (\ref{kim}) is not available in general, the need to numerically approximate the early exercise boundary and also the price of the option appears naturally.
Generally, the numerical methods used for solving the integral equations describing the early exercise boundary in the American option pricing literature could be classified into three main categories:
\begin{description}
\item[Direct Quadrature Methods:]  This family of methods could be considered as the oldest approximation schemes for integral equations which approximate the integral terms by numerical quadrature rules such as trapezoidal, midpoint and Simpson rules for equidistant meshes or Gaussian type quadrature rules \cite{hack}. In the special case of integral equations arising from American option pricing, we could construct a system of nonlinear equations with the solution $\mathcal{B}(t_{i})$ (for given $t_{i}$'s, $i=1,2,\cdots,n$) and then solve these equations to finally arrive at  a global solution using the theory of polynomial or rational interpolation. The first idea of this kind is due to Huang \textit{et al.} \cite{huang1996pricing} which is pursued later by Kallast and Kivinukk \cite{kallast} who focus on Kim's integral representation for the early exercise boundary and apply a suitable quadrature rule based on Sullivan's idea \cite{sullivan} accompanied by the Newton-Raphson method in order to obtain a fast numerical approach. Heider \cite{heider} has also employed an integral transform to propose a Nystr\"{o}m-type discretization for Kim's integral equation.
\end{description}
\begin{description}
\item[Successive Iteration Methods:] In this method, we construct a recursive sequence of the form $\mathcal{B}^{(k+1)} = F (\mathcal{B}^{(k)}), k=0,1,2,\cdots$ in which $F$ is a fully nonlinear integral operator with fixed point $\mathcal{B}$. For the one dimensional Kim's integral equation, the method of fixed point iteration with Gauss-Kronrod rule has been used by Kim \cite{kim2}. Recently,  a modified Newton iterative solution that operates in parallel is obtained for the approximation of the early exercise boundary by Cortazar et al. in \cite{cortazar}. Also, this approach has been employed for other nonlinear integral equations in the corresponding literature \cite{lauko, shev2}.
\end{description}
\begin{description}
\item[Collocation-Based Methods:] Classically, collocation discretization which is based on interpolatory projection of $C(X)$ (the space of continuous functions on $X$) onto a finite-dimensional subspace is widely used in the numerical solution of integral and differential equations. Noting that the $\mathcal{B}(t)$ term in the integral equation for the early exercise boundary appears inside the logarithm in (\ref{d1d2}), it is suitable to define an approximation by multi-piece exponential functions (see Ju \cite{ju} for more details). Aitsahlia \cite{lai2} replaced piecewise exponential functions with linear splines to improve Ju's approach and to get more accuracy and speed-up gains.  Recently, a polynomial spectral collocation method for computing the American call and put option prices has been considered in \cite{andersen} based on Kim's integral equation.
\end{description}
What is the key point in the numerical treatment of nonlinear integral equations discussed above is the weakly singular character of these equations and the resulting singular behavior of the exercise boundary. In this respect, we propose a product integration method which belongs to the first category of numerical schemes based on rational barycentric interpolation to overcome this difficulty in the discretization of the integral terms.
\subsection{Product Integration Method}\label{produc}
Nystr\"{o}m method is one of the popular ways for numerical solution of integral equations \cite{atkinson}. It should be noticed that in the literature of Volterra integral equations, this method is called quadrature method \cite{hack}, however both of them use the same plan to approximate the solution.  Product integration method is a kind of Nystr\"{o}m method which is utilized to numerically solve weakly singular integral equations.  In this method, the smooth part of the kernel is interpolated in order to manage the weak singularity of the kernel \cite{cuminato, orsi,  hoog}.

In order to present the principles underlying the method, we first choose $n+1$ distinct points, $\{ t_{i}\}_{i=0}^{n}$ in the interval $[0, T]$ and then collocate (\ref{asli}) at these nodes to obtain
\begin{equation}\label{collocation}
{u}(t_{i}) = g(t_{i}, {u}(t_{i})) + \int_{0}^{t_{i}}\frac{1}{\sqrt{t_{i}-s}}k_{1}(t_{i}, s, {u}(t_{i}), {u}(s))\mathrm{d}s+\int_{0}^{t_{i}}k_{2}(t_{i},s ,{u}(t_{i}),{u}(s))\mathrm{d}s,
\end{equation}
for $i=0,1,\cdots, n$. Based on the fact that the second integral term has a smooth kernel, it is utilized by direct quadrature rule. Moreover,
the interpolation of the smooth part of the  first kernel, $k_{1}$, is used to cope with the weakly singular term, (see e.g. \cite{orsi, hack, orsi1996product}).
In this respect, we project the functions \[K_{i}(s, {u}(s)) : = k_{1}(t_{i}, s, {u}(t_{i}), {u}(s)), \quad i=0,1, \cdots, n, \] into the space $V_{n}= \textmd{Span}\{\mathcal{L}_{j}(s)\}_{j=0}^{n}$
for appropriate basis functions $\mathcal{L}_{j}(s), j=0,1,\cdots,n$ to obtain
 \begin{equation*}
(\mathcal{P}_{n}K_{i})(s) = \sum_{j=0}^{n}K_{i}(t_{j}, {u}(t_{j}))\mathcal{L}_{j}(s).
\end{equation*}
Now, the above approximation of the kernel is substituted into (\ref{collocation}) and the following system of equations is obtained
\begin{equation}\label{dis}
u_{i} =  g(t_{i}, u_{i}) + \underbrace{\sum_{j=0}^{i} w_{i,j} k_{1}(t_{i}, t_{j}, {u}_{i}, {u}_{j})}_{\text{product integration}}+ \underbrace{\sum_{j=0}^{i} \omega_{j} k_{2}(t_{i}, t_{j}, {u}_{i}, {u}_{j})}_{\text{direct quadrature}}, \quad  i=0, 1, \cdots, n,
\end{equation}
where $w_{i, j} = \int_{0}^{t_{i}} \frac{\mathcal{L}_{j}(s)}{\sqrt{t_{i}-s}}\mathrm{d}s$ and $\omega_{j}$'s are the quadrature weights.

In practice, the weights $w_{i, j}$ and $\omega_{j}$ should be computed numerically by an efficient quadrature rule with rapid convergence, such as Gauss-Legendre or Clenshaw-Curtis method.
When the weights are obtained, the approximate solutions, $u_{i} \approx {u}(t_{i})$, are computed as the solution of the nonlinear system of equations (\ref{dis}).

In the implementation of the product integration method, there are two crucial points which should be taken into account: one is choosing a finite dimensional subspace  $V_{n}$ of $C([0,T])$ and the second point is the numerical quadrature used in the discretization of the integrals. A natural and available choice for these aims is to interpolate the kernel with the Lagrange polynomials and use interpolatory quadrature rules. It is now a well-known fact that barycentric form of interpolation is a viable variant of Lagrange's classic polynomial interpolation which has desirable features such as stability and computational speed \cite{kleinthesis, bary}. In the sequel, we present a brief overview of barycentric interpolation and quadrature methods.
\subsection{Barycentric Interpolation}\label{baryin}
Let $\{t_{i}\}_{i=0}^{n}$ be a set of strictly ordered equidistant nodes in $[0, T]$ with a fix grid spacing $h$. The barycentric interpolation of the data values $ \{(t_{i}, f(t_{i}))\}_{i=0}^{n} $ could be written as
\begin{equation}\label{barycent}
\mathcal ({\mathcal{P}}_{n}f)(t)=\frac{\sum_{i=0}^{n}\frac{\beta_{i}}{t-t_{i}}f(t_{i})}{\sum_{i=0}^{n}\frac{\beta_{i}}{t-t_{i}}}
= \sum_{i=0}^{n}f(t_{i})\mathcal{L}_{i}(t),
\end{equation}
in which \begin{equation}\label{li}
\mathcal{L}_{i}(t) =\frac{\frac{\beta_{i}}{t-t_{i}}}{\sum_{i=0}^{n}\frac{\beta_{i}}{t-t_{i}}},\quad i=0,1,\cdots, n.\end{equation}
In the case of Lagrange interpolation, the weights $ \beta_{i} $ are given by
\begin{equation}\label{weibary}
\beta_{i}=\frac{1}{\Pi_{i\neq j}(t_{i} - t_{j})}, \quad i=0,1, \cdots, n, \end{equation}
but if we choose other weights in the above expression, then the resulting function $({\mathcal{P}}_{n}f)(t)$ still interpolates the data $f$ even though it is no longer in general a polynomial \cite{bary}.

Among the most important alternative options for the $\beta_{i}$'s, we could mention the Berrut's weights given by
\begin{equation}\label{ber}
\beta_{i}=(-1)^{i}, \quad i=0,1, \ldots, n,
\end{equation}
in which $n$ is an odd number \cite{kai112}. It could be shown (see e.g. \cite{bary}) that with the above weights, the resulting interpolator is a rational function with no poles in the interval of interpolation and the order of convergence is $\mathcal{O}(\frac{1}{n})$.

Investigations in this area to obtain some weights which will produce interpolants, $ \mathcal{P}_{n}f $, with no poles and good approximation
properties have led to the family of linear barycentric
rational interpolations introduced by Floater and Hormann \cite{kai}.
Let for a fixed integer $ 0\leq d \leq n, $ the polynomials $ \{p_{i}(t) \}_{i=0}^{n-d} $ interpolate $ f $ at the nodes
$ \{ t_{i}, \ldots, t_{i+d}\} $. Then we could write
\begin{equation}\label{rational}
(\mathcal{P}_{n}f)(t)=\frac{\sum_{i=0}^{n-d} \lambda_{i}(t)p_{i}(t)}{\sum_{i=0}^{n-d} \lambda_{i}(t)},
\end{equation}
where
\begin{equation*}
\lambda_{i}(t)= \frac{(-1)^{i}}{(t-t_{i})\ldots (t-t_{i+d}) }.
\end{equation*}
Eq. (\ref{rational}) can be rewritten in the barycentric form (\ref{barycent}) with the weights
\begin{equation}\label{baryrashi}
\beta_{i}=(-1)^{i-d}\sum_{j\in J_{i}} \binom {d} {i-j},
\end{equation}
where $ J_{i} $ is defined as
\[ J_{i}=\{ \max (1, i-d) \leq j \leq \min (i, n-d-1) \}.\]
Rational barycentric interpolation with the weights (\ref{baryrashi}) has a superior advantage compared to other forms of the barycentric interpolation as the following theorem shows:
\begin{theorem}\label{inregh} (Floater and Hormann, \cite{kai}) Suppose that $d\geq 1$ and $f \in C^{d+2}([0, T])$. If $n-d$ is odd, then
\[
\Vert f - \mathcal{P}_{n}f \Vert_{\infty} \leq h^{d+1} T\frac{\Vert f^{(d+2)}\Vert_{\infty}}{d+2},
\]
if $n-d$ is even, then
\[
\Vert f - \mathcal{P}_{n}f \Vert_{\infty} \leq h^{d+1}\Big( T\frac{\Vert f^{(d+2)}\Vert_{\infty}}{d+2} + \frac{\Vert f^{(d+1)}\Vert_{\infty}}{d+1}\Big).
\]
\end{theorem}
\subsection{Barycentric Rational Quadrature}\label{quada123}
In this subsection, an equivalent interpolatory quadrature based on rational interpolation is introduced.
Barycentric quadrature and its features have been studied extensively in \cite{kleinthesis, bary}.
The linear interpolant \[ (\mathcal{P}_{n}f)(t)=\sum_{i=0}^{n}f(t_{i})\mathcal{L}_{i}(t), \]
naturally leads to the following classical quadrature formula
\begin{equation}\label{2.10}
 {Q}_{n}[f] = \sum_{i=0}^{n}\omega_{i,n}f(t_{i}),
\end{equation}
where the corresponding quadrature weights, $ \omega_{i,n} $,  are defined by
\begin{equation}\label{quadwei}\omega_{i,n}=\int_{0}^{T}\mathcal{L}_{i}(t)\mathrm{d}t, \quad i=0, \ldots, n.\end{equation}
The stability condition of the quadrature method is given by (see \cite{hack2})
\[\sup \Big\{ \sum_{i=0}^{n}\vert \omega_{i,n} \vert, n\in \Bbb N\Big\} < \infty.\]
It  follows  from (\ref{quadwei}) that
\begin{equation}\label{123}
\sum_{i=0}^{n}\vert \omega_{i,n}\vert \leq \int_{0}^{T}\sum_{i=0}^{n}\bigg\vert \dfrac{\frac{\beta_{i}}{t-t_{i}}}{\sum_{i=0}^{n}\frac{\beta_{i}}{t-t_{i}}} \bigg{\vert}\mathrm{d}t = \int_{0}^{T}  \Lambda_{n}(t)\mathrm{d}t,
\end{equation}
where the function
$\Lambda_{n}(t)=\sum_{i=0}^{n}\vert \mathcal{L}_{i}(t) \vert$
is the \textit{Lebesgue function} and
\begin{equation}\label{leb}\Lambda_{n}=\sup_{t\in [a,b]}
\Lambda_{n}(t),\end{equation}
is the \textit{Lebesgue constant} \cite{bary}.

By this relation, the following upper bound could be obtained for (\ref{123})
\begin{equation}
\sum_{i=0}^{n}\vert \omega_{i,n}\vert \leq T\Lambda_{n},
\end{equation}
so the stability of the direct quadrature method depends on the stability of the interpolation process.
The Lebesgue constant for Lagrange interpolation at equaidistant nodes grow exponentially
 \[\Lambda_{n}\approx \frac{2^{n+1}}{n\log(n)}, \quad n \rightarrow \infty,\]
 as presented in \cite{bary}.
 It is shown this value associated with the family of Floater-Hormann interpolant with $d\geq 1$ grows logarithmically as demonstrated by the following theorem:
 \begin{theorem}(Bos et al., \cite{bos})
 The Lebesgue constant associated with rational interpolation at equidistant nodes with basis functions (\ref{li}) associated with coefficients
(\ref{baryrashi}) satisfies
\[\Lambda_{n} \leq 2^{d-1}\Big(2+\log(n)\Big).\]
 \end{theorem}
The following theorem gives an upper bound for the linear barycentric rational quadrature.
\begin{theorem}\label{quad}(Klein, \cite[Theorem 4.1]{kleinthesis})
Suppose $n$ and $d$ with $d\leq n$ are positive integers,  $ f \in C^{d+2}[a, b] $ and $ \mathcal{P}_{n}f $ is the rational interpolant with parameter $d$ given by (\ref{rational}). Let the quadrature weights (\ref{baryrashi}) be approximated by a quadrature rule which convergence at least at the rate $ \mathcal{O}(h^{d+1}) $ and degree of precision at least $d+1$. Then
	\begin{equation}
	\Big\vert \int_{a}^{b}f(t)\mathrm{d}t - \sum_{i=0}^{n}\omega_{i,n}f_{i} \Big\vert \leq C h^{d+1},
	\end{equation}
where $ C $ is a constant depending on $ d $, derivatives of $ f $ and the length of the interval.
\end{theorem}
\subsection{Approximation of the Early Exercise Boundary}
In this subsection, we first review some regularity properties of the early exercise boundary and then based on the previous tools, we discretize the nonlinear integral equation to obtain an approximation for $ \mathcal{B}(t) $. Finally an error analysis for the proposed method will be presented.
\begin{theorem}(Karatzas \textit{et al. }\cite{karatzas1998methods})\label{reg}
	Let $ \mathcal{B}(t) $ be the early exercise boundary of the American put price. Then it is a continuously differentiable function on $ (0,T] $ and
	\begin{equation}\label{w}
	\begin{split}
	\lim_{s\rightarrow 0} \mathcal{B}(t)& = \mathcal{B}(0) = K, \quad \delta \leq r,\\
	\lim_{s\rightarrow 0} \mathcal{B}(t)& = \mathcal{B}(0) =(\frac{r}{\delta}) K, \quad \delta > r.
	\end{split}
	\end{equation}
\end{theorem}
We now discretize Eq. (\ref{kimnondiv}) using the product integration method to arrive at
\begin{equation}\label{diskimnondiv}
\begin{split}
\mathcal{B}_{i}\aleph(d_{1}(\mathcal{B}_{i},t_{i}, K)) +& \mathcal{B}_{i} \frac{1}{\sigma\sqrt{2\pi t_{i}}}K\exp\Big( -\frac{1}{2}d_{1}(\mathcal{B}_{i},t_{i}, K)^{2}\Big) \\
= &\frac{1}{\sigma \sqrt{2\pi t_{i}}}K\exp\Big(-\Big[rt_{i} + \frac{1}{2} d_{2}(\mathcal{B}_{i},t_{i}, K)^{2}\Big]\Big)\\
&+ \frac{rK}{\sigma\sqrt{2\pi}}\sum_{j=0}^{i}w_{i,j} \exp\Big(-(r(t_{i}-t_{j})+\frac{1}{2}d_{2}(\mathcal{B}_{i}, t_{i}-t_{j}, \mathcal{B}_{j})^{2})\Big), \end{split}
\end{equation}
in which $\omega_{i,j} = \int_{0}^{t_{i}} \frac{\mathcal{L}_{j}(s)}{\sqrt{t_{i}-s}}\mathrm{d}s$ and $\mathcal{L}_{j}(s)$ is defined as in (\ref{li}) with the coefficients (\ref{ber}) or (\ref{baryrashi}). A similar expression could be obtained for Eq. (\ref{kimdiv}) by the product integration and also direct quadrature methods:
\begin{eqnarray}\label{diskimdiv}
-\mathcal{B}_{i}\exp(-\delta t_{i}) \aleph \Big(d_{1}(\mathcal{B}_{i},t_{i}, K)\Big) +\frac{K}{\sigma \sqrt{2\pi t_{i}}} \exp \Big(- (rt_{i} + \frac{1}{2}d_{2}(\mathcal{B}_{i}, t_{i}, K)^{2})  \Big)\\\notag
 -\frac{\mathcal{B}_{i}}{\sigma \sqrt{2\pi t_{i}}} \exp \Big(- (\delta t_{i} + \frac{1}{2}d_{1}(\mathcal{B}_{i}, t_{i}, K)^{2})  \Big)\\\notag
+
\frac{1}{\sigma \sqrt{2\pi}}\sum_{j=0}^{i} w_{i,j}\Big[ rK \exp \Big(- r(t_{i}-t_{j}) - \frac{1}{2} d_{2}(\mathcal{B}_{i}, t_{i}-t_{j}, \mathcal{B}_{j})^{2}\Big)\\\notag
-\delta \mathcal{B}_{i} \exp \Big(- \delta(t_{i} - t_{j}) - \frac{1}{2} d_{1}(\mathcal{B}_{i}, t_{i}-t_{j}, \mathcal{B}_{j})^{2}\Big)\Big]\\\notag
- \delta \mathcal{B}_{i} \sum_{j=0}^{i}\omega_{j}\exp \Big(-\delta (t_{i} - t_{j})\Big)  \aleph \Big(d_{1} (\mathcal{B}_{i}, t_{i}- t_{j},  \mathcal{B}_{j})\Big) = 0.
\end{eqnarray}
As soon as $\mathcal{B}_{i}$'s are obtained from the above equations, we could employ the barycentric rational interpolation to obtain a continuous approximating function
\begin{equation}\label{ghe}
\mathcal{B}_{n}(t) = \sum_{i=0}^{n}\mathcal{B}_{i}\mathcal{L}_{i}(t).\end{equation}

In the following, we give an error bound for discretization process obtained via (\ref{diskimdiv}). It must be mentioned that for Eq. (\ref{diskimnondiv}) a similar result could be obtained.
\begin{lemma}\label{bn}
Let $\mathcal{B}(t)$ be the exact solution of Eq. (\ref{kimdiv}) and $\mathcal{B}_{n}(t)$ be given by (\ref{ghe}). Then there exists a positive constant $C$ independent of $n$ such that
\begin{equation*}
\Vert \mathcal{B} - \mathcal{B}_{n} \Vert_{\infty} \leq C\log(n)h^{d+1}.
\end{equation*}
\end{lemma}
\begin{proof}
Using the triangle inequality, we arrive at
\begin{equation}\label{3.19}
\begin{split}
\Vert \mathcal{B} - \mathcal{B}_{n} \Vert_{\infty} & =\Vert \mathcal{B} - \mathcal{P}_{n} \mathcal{B}+ \mathcal{P}_{n} \mathcal{B} - \mathcal{B}_{n} \Vert_{\infty} \\
 &  \leq \Vert \mathcal{B} - \mathcal{P}_{n} \mathcal{B} \Vert_{\infty} + \Vert  \mathcal{P}_{n} \mathcal{B}-\mathcal{B}_{n} \Vert_{\infty},
  \end{split}
\end{equation}
in which $\mathcal{P}_{n}$ is the interpolation operator defined in (\ref{barycent}). The first term in the right hand side of (\ref{3.19}) is the interpolation error which by Theorem \ref{inregh}, its rate of convergence is $\mathcal{O}(h^{d+1})$. Also, the second term could be bounded for each $t \in (0,T]$ as
\begin{equation}
\begin{split}
\vert  (\mathcal{P}_{n}\mathcal{B})(t) - \mathcal{B}_{n}(t)\vert &= \Big \vert \sum_{i=0}^{n} \mathcal{L}_{i}(t) (\mathcal{B}(t_{i}) -\mathcal{B}_{i}) \Big \vert \\
 & \leq \sum_{i=0}^{n} \vert \mathcal{L}_{i}(t) \vert \vert \mathcal{B}(t_{i}) - \mathcal{B}_{i}\vert,\\
\end{split}
\end{equation}
and so
 \begin{center}
 \begin{equation}\label{uppad}
  \Vert  \mathcal{P}_{n} \mathcal{B}-\mathcal{B}_{n} \Vert_{\infty}  \leq  \Lambda_{n} \max_{i} \{\mathcal{B}(t_{i}) - \mathcal{B}_{i}\}.\end{equation}
 \end{center}
Notice that if we collocate Eq. (\ref{kimdiv}) at the grid points, it could be seen that  $\mathcal{B}(t_{i})$ is the exact solution of the obtained equation. Based on this fact and using  Eq. (\ref{diskimdiv}), we see  that the upper bound for $ \max_{i} \{\mathcal{B}(t_{i}) - \mathcal{B}_{i}\}$ depends on the interpolation and numerical quadrature errors. Due to the smoothness of the functions $\exp(.)$ and $\aleph(.)$ inside the equation and using Theorem \ref{inregh}, an error of order  $\mathcal{O}(h^{d+1})$ is achieved in the collocation procedure. On the other hand,  the Lebesgue constant $\Lambda_{n}$ is bounded by the term  $2^{d-1}(2+\log(n))$, so the final result is given by the  Theorem \ref{quad}.
\end{proof}
\section{Approximation of the American Option Price}\label{thistable}
In this section, the pricing of an American put option will be considered. Note that the price of the corresponding American call could be found by put-call symmetry \cite{andersen}. It could easily be seen that as soon as the early exercise boundary is determined, the option price could then be obtained by employing an appropriate quadrature rule applied to the integral terms in Eq. (\ref{price}).

 For this purpose and due to the complexity of the kernel, we utilize the quadrature method introduced in Subsection \ref{quada123} to approximate the price. In the reminder, we analyze the approximation order of the proposed  quadrature method in Theorem \ref{errores}.

Before that, we introduce the notations $P_{n}(t, S)$ and $\tilde{P}_{n}(t, S)$, defined respectively by
\begin{equation}\label{PN}
\begin{split}
P_{n}(t, S)=~&p(t, S) + \int_{0}^{t}rK e^{-r(t-\xi)}\aleph (-d_{2}(S, t -\xi, \mathcal{B}_{n}(\xi)))\mathrm{d}\xi \\
 &  - \int_{0}^{t} \delta S e^{-\delta (t- \xi)} \aleph (-d_{1}(S, t- \xi, \mathcal{B}_{n}(\xi)))  \mathrm{d}\xi,
  \end{split}
\end{equation}
\begin{equation}\label{PN2}
\begin{split}\tilde{P}_{n}(t, S) = ~&p(t, S) + \sum_{i=0}^{n}rK e^{-r(t - t_{i})}\aleph \Big(-d_{2}(S, t - t_{i}, \mathcal{B}_{n}(t_{i}))\Big) \\
 &  - \sum_{i=0}^{n} \delta S e^{-\delta (t- t_{i})} \aleph \Big(-d_{1}(S, t - t_{i}, \mathcal{B}_{n}(t_{i}))\Big).\end{split}
\end{equation}
In both formulae, $\mathcal{B}_{n}(\xi)$ is the approximant of the early exercise boundary obtained as (\ref{ghe}).

Let us consider the price representation (\ref{price}) as a nonlinear operator
 \begin{equation}\label{nonop}
 \begin{split}
 P :  C\big(  (0, \infty) \big)  \rightarrow & C \left((0, T]\times (0, \infty)\right) \\
 \mathcal{B}\mapsto & P(\mathcal{B}) =  P(t,S).
 \end{split}
 \end{equation}
In the following lemma,  the Fr\'{e}chet derivative of this nonlinear operator is given explicitly.
  \begin{lemma} (Heider, \cite{heider2007condition})
 The Fr\'{e}chet derivative of the nonlinear operator (\ref{nonop}) at $\mathcal{B}(t)$ is given by
  \begin{equation}
  \begin{split}
  (P'(\mathcal{B})h) (t,S) =&  \frac{rK}{\sigma \sqrt{2\pi}}\int_{0}^{t} \frac{e^{-r (t-\xi)}}{\mathcal{B}(\xi) \sqrt{t-\xi}} e^{- \frac{d_{2}(S, t-\xi, \mathcal{B}(\xi))^{2}}{2}} h(\xi) \mathrm{d}\xi \\
  - & \frac{\delta S}{\sigma \sqrt{2\pi}}\int_{0}^{t} \frac{e^{-\delta (t-\xi)}}{\mathcal{B}(\xi) \sqrt{t-\xi}} e^{- \frac{d_{1}(S, t-\xi, \mathcal{B}(\xi))^{2}}{2}} h(\xi) \mathrm{d}\xi.
  \end{split}
  \end{equation}
 \end{lemma}
 \begin{theorem}\label{errores}
Let $P(t, S)$ be the price of an American put option with the parameters defined in Section \ref{option}. Futhermore assume that $\mathcal{B}(t)$ denotes its early exercise boundary function. Let also $\tilde{P}_{n}(t, S)$ be an approximation of $P(t, S)$.
 Then we have
 \[\vert P(t, S) - \tilde{P}_{n}(t, S) \vert \leq \frac{\theta -1}{\sigma \theta \sqrt{2}} \left( \frac{\sqrt{\delta}S}{K} + \sqrt{r}\right)
C\log(n)h^{d+1},  \]
 where
 \[ \theta = \frac{- (r-\delta- \frac{1}{2}\sigma^{2}) - \sqrt{(r-\delta- \frac{1}{2}\sigma^{2})^{2} + 2 \sigma^{2}r}}{ \sigma^{2}}. \]
 \end{theorem}
\begin{proof}
The triangle inequality gives
\[
\vert P(t, S) - \tilde{P}_{n}(t, S) \vert \leq \vert P(t, S) - P_{n}(t, S) \vert + \vert P_{n}(t, S) - \tilde{P}_{n}(t, S) \vert.
 \]
Now applying the mean value theorem for operators (see e.g. Proposition 5.3.11 in \cite{atkinson}), we obtain:
 \begin{equation}\label{frechetterm}
 \vert P(t, S) - \tilde{P}_{n}(t, S) \vert\leq \sup_{0 \leq \lambda \leq 1} \Vert P'((1- \lambda)\mathcal{B} + \lambda \mathcal{B}_{n}) \Vert_{\infty} \Vert \mathcal{B} - \mathcal{B}_{n} \Vert_{\infty},\end{equation}
in which $P'$ is the Fr\'{e}chet derivative derived in Lemma \ref{nonop}.  It could easily verified that for $a> 0$ we have
\begin{equation}\label{1e} a \int_{0}^{t} \frac{e^{-a(t- \xi)}}{\sqrt{t -\xi}}\mathrm{d}\xi = \sqrt{a \pi} \erf (\sqrt{a t}).\end{equation}
Furthermore, the monotonicity of $\mathcal{B}$  and Theorem \ref{reg} gives
 \begin{equation}\label{2e} (1-\theta)\mathcal{B}(\xi) + \theta \mathcal{B}_{n}(\xi) \geq \mathcal{B}(0^{+}). \end{equation}
Moreover, it could be shown (see e.g. \cite{kim}) that
 \begin{equation}\label{3e}\frac{\theta K}{\theta -1 } \leq \mathcal{B}(t) \leq \mathcal{B}(0^{+}). \end{equation}
 So by the relations (\ref{1e} - \ref{3e}), the supremum term in (\ref{frechetterm}) could be bounded by $\frac{\theta -1}{\sigma \theta \sqrt{2}} \left( \frac{\sqrt{\delta}S}{K} + \sqrt{r}\right)$ (for more detail see Proposition 3.1 in \cite{heider2007condition}).

Also, an upper bound could be obtained for  $\vert P_{n}(t, S) - {P}(t, S) \vert$ by considering Eqs. (\ref{PN}) and (\ref{PN2}) and the  Theorem \ref{quad}. The final result now could be obtained from Lemma \ref{bn}.
\end{proof}
\section{Numerical Experiments}\label{NE}
In this section, we give some numerical evidence concerning accuracy and the rate of convergence of the presented method in this paper. In this respect, we compute the early exercise boundary as well as the option price for a set of test problems chosen from the literature (see e. g. \cite{ju, kallast}). We also compare our results with a number of alternative approaches, some of them based on integral equation representations and the others belonging to the semi-analytical family of methods.

In the reminder, we denote by FH($d$) the product integration method based on linear barycentric rational interpolation using Floater-Hormann weights of degree $d$ (introduced in Subsection \ref{baryin}).
The combination of Berrut and Floater-Hormann weights (see respectively  (\ref{ber}) and (\ref{baryrashi})) is used to compute the early exercise boundary of an American put option which is denoted by BFH($d$) in the sequel.

In order to solve the system of equations (\ref{dis}), a natural idea is to utilize the Newton method which is a popular choice\footnote{By using the \texttt{fsolve} command in {MATLAB}$^{\circledR}$ environment.} in the corresponding literature \cite{brunner}. But due to the complexity of the kernel and forcing functions, computing such a nonlinear scheme may lead to a potentially time consuming procedure involving sequential iterative linearization.

In this respect, along with the Newton iteration, we also propose a hybrid ``Newton-interpolation scheme'' which solves the system of equations by Newton method based on a small number of grid points and then interpolates the results linearly between the nodes.
More precisely, we distribute $m-2$ points in the interval $[t_{i}, t_{i+1}]$ and  recover $\{\mathcal{B}(t_{i,j})\}_{j=2}^{m-1}$ by using the linear interpolant from (\ref{dis}). This approach combined with Berrut-Floater-Hormann and Floater-Hormann schemes will be denoted by BFH($d$, $m$) and FH($d$, $m$), respectively in the reminder. Also in this case, the total number of grid points will be  $N= n+(n-1)(m-2)$.

The proposed algorithms are implemented in {MATLAB}$^{\circledR}$ on a PC with 4.00 GHz Intel$^{\circledR}$ Core\textsuperscript{TM} i7 dual processor with 16 GB RAM.
We report our results for the early exercise boundary, $\mathcal{B}(t)$ and also the American put value $P(T,S)$ with the parameter set $( K, T, r, \sigma) = (100, 3,  0.08, 0.2)$ and with the dividend yeilds
$\delta \in \{0, 0.04, 0.08, 0.12\}$
for $n=64$ and $d=3$ in Figures \ref{boundary} and \ref{putvalue}, respectively.  In Figure \ref{putvalue}, the dotted lines show the exact put values obtained from the binomial tree model (BIN) with $n = 10,000$ time steps which will be used as the benchmarks in each case.

We also have prepared Table \ref{numericalresults} which shows the absolute error of the results and a comparison between the studied test cases. This table confirms that BFH($2$) gives a better result in comparison with the other reported cases. It must be noticed that the columns KJK which utilizes a fixed point method and also BFH($2$) method, both are based on the approximation of the same integral equation.
\begin{figure}[ht]
\begin{center}
\includegraphics[width=10.01cm, height=5cm]{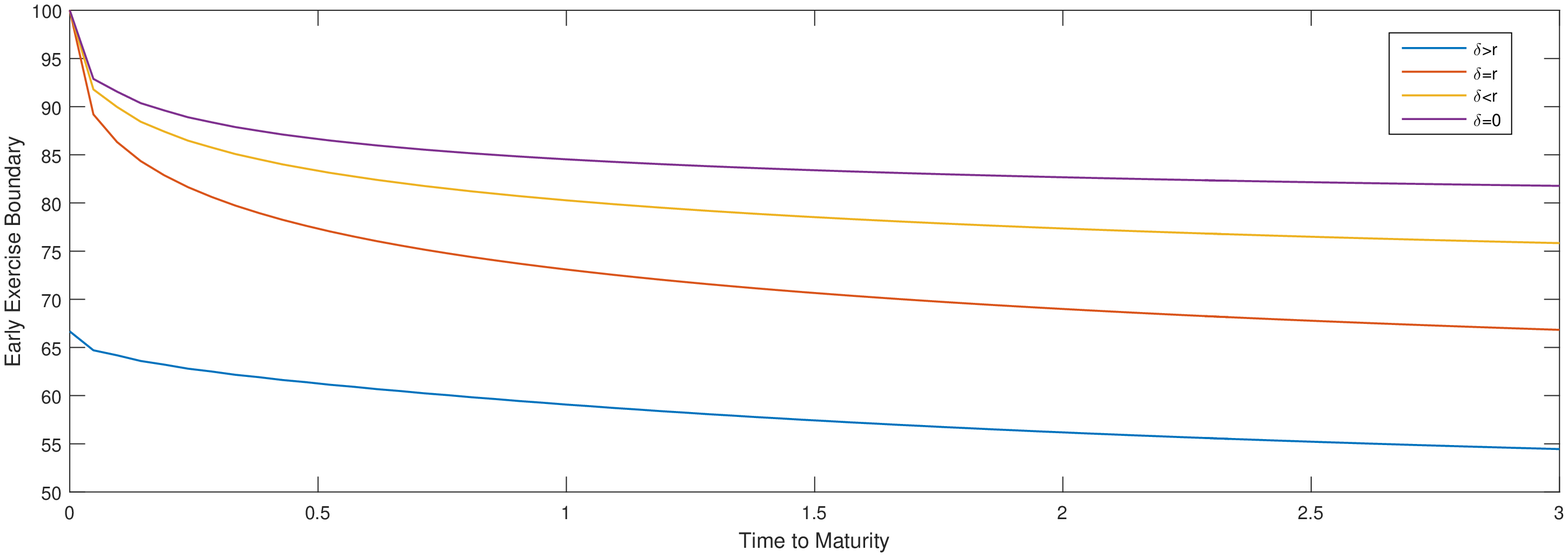}
\caption{The early exercise boundary of an American put obtained from FH($3$) method for $n=64$ and $\delta \in \{0, 0.04, 0.08, 0.12\}.$}
\label{boundary}
\end{center}
\end{figure}
\begin{figure}[ht]
	\begin{center}
		\includegraphics[width=10.01cm, height=5cm]{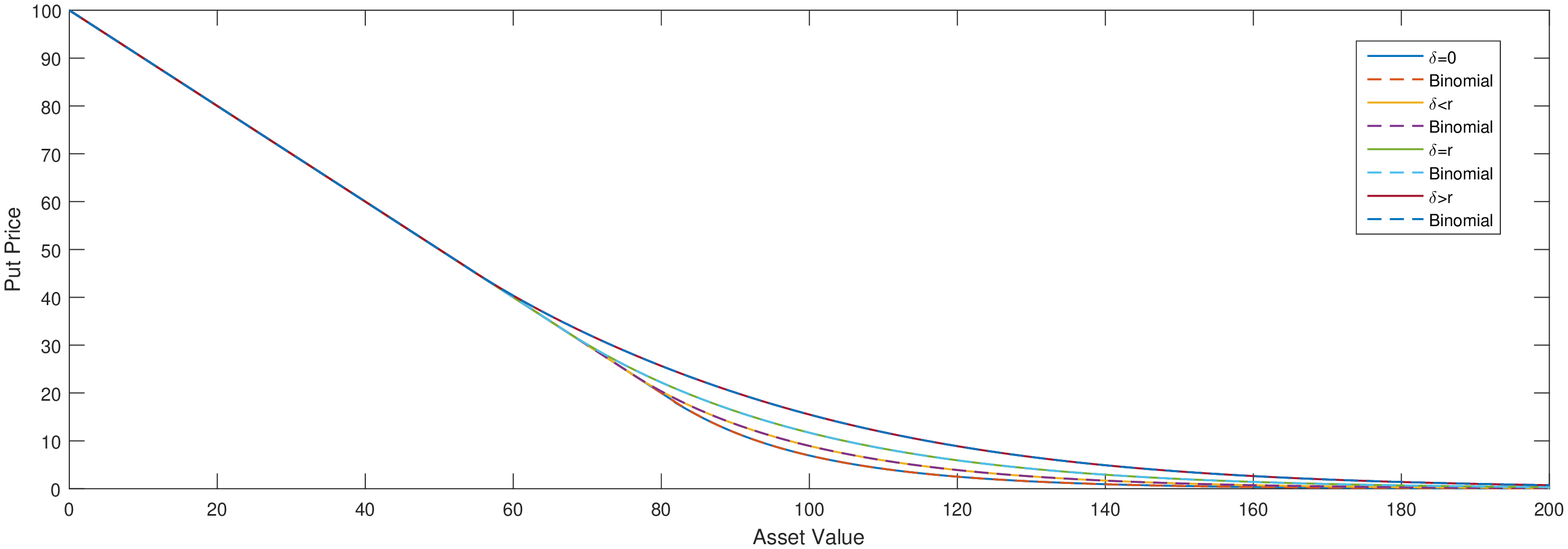}
		\caption{The put value $P(T, S)$ for $S= 120$, $n=64$ and  $\delta \in \{0, 0.04, 0.08, 0.12\}$. }
		\label{putvalue}
	\end{center}
\end{figure}
\begin{table}[ht!]
\centering
\resizebox{\textwidth}{!}{
\begin{tabular}{l c |  l l l l l l l l}

$S$ & BIN & GJ4  & MGJ2 & LUBA &EXP3 & KJK &KK &BFH($2$)\\[0.5ex]

\hline
 &$22.2050$&$22.2079$ & $22.7106$ & $22.1985$ & $22.2084$ & $22.1942$ &  $22.1900$&$22.2048$  \\[-1ex]
\raisebox{1.5ex}{$80$}
&- &$2.9$\text{e}$-03$ & $5.1$\text{e}$-01$ & $6.5$\text{e}$-03$ & $3.4$\text{e}$-03$& $1.1$\text{e}$-02$ & $1.5$\text{e}$-02$& $2.0$\text{e}$-04$ \\[1ex]
\hline

 &$16.2071$&$16.1639$ & $16.5205$ & $16.1986$ & $16.2106$ & $16.1999$ &  $16.1960$&$16.2068$  \\[-1ex]
\raisebox{1.5ex}{$90$}
&-& $4.3$\text{e}$-02$ & $3.6$\text{e}$-01$ & $5.9$\text{e}$-02$ & $7.2$\text{e}$-03$& $1.1$\text{e}$-02$ & $3.9$\text{e}$-03$& $1.1$\text{e}$-04$ \\[1ex]
\hline

 &$11.7037$ &$11.7053$ & $11.8106$ & $11.6988$ & $11.7066$ & $11.6991$ &  $11.6958$&$11.7037$  \\[-1ex]
\raisebox{1.5ex}{$100$}
& -&$1.6$\text{e}$-03$ & $1.1$\text{e}$-01$ & $4.9$\text{e}$-03$ & $2.9$\text{e}$-03$& $4.9$\text{e}$-03$ & $7.9$\text{e}$-03$& $1.0$\text{e}$-05$ \\[1ex]
\hline

 &$8.3671$ &$8.3886$ & $8.4072$ & $8.3630$ & $8.3695$ & $8.3638$ &  $8.3613$&$8.3669$  \\[-1ex]
\raisebox{1.5ex}{$110$}
&-& $2.1$\text{e}$-02$ & $4.0$\text{e}$-02$ & $4.1$\text{e}$-03$ & $2.4$\text{e}$-03$& $3.3$\text{e}$-03$ & $5.8$\text{e}$-03$& $2.0$\text{e}$-04$ \\[1ex]
\hline
&$5.9299$ &$5.9435$ & $5.9310$ & $5.9261$ & $5.9323$ & $5.9278$ & $5.9258$&$5.9298$  \\[-1ex]
\raisebox{1.5ex}{$120$}
&-& $1.4$\text{e}$-02$ & $1.1$\text{e}$-03$ & $3.8$\text{e}$-03$ & $2.4$\text{e}$-03$& $2.1$\text{e}$-03$ & $4.1$\text{e}$-03$& $1.0$\text{e}$-04$ \\[1ex]
\end{tabular}}
\caption{
Estimated $3$-year put option values by BFH($2$) for $K = 100 $ and $S $ as listed in the  last column of the table. The parameter
set used are $r = \delta = 0.08$ and $\sigma = 0.2$ and $n=32$. The other columns are respectively
BIN:  the binomial tree model with $n=10000$ time steps;
GJ4: the four-point extrapolation scheme of Geske and Johnson \cite{geske1984american};
MGJ2:  the modified two-point Geske and Johnson method of Bunch and Johnson \cite{bunch1992simple};
LUBA:  the lower and upper bound approximation of Broadie and Detemple \cite{broadie1996american};
EXP3: the multi-piece exponential functions method of Ju \cite{ju} using the three-point Richardson extrapolation;
KJK:  the iteration method of Kim  \textit{et al.} \cite{kim2};
KK:  the trapezoidal formulas approximations of Kallast and Kivinukk accompanied by the Newton-Raphson iteration \cite{kallast}.}
\label{numericalresults}
\end{table}

In order to gain some insight into the efficiency of FH($d$), BFH($d$), FH($d$, $m$) and BFH($d$, $m$) methods we have reported work-precision diagrams for the proposed methods in Figures \ref{fig:test 1}-\ref{modifiedcom}. As it is expected, using more nodes will lead to more time to obtain the approximate solution with a different rate in each case. Figures \ref{fig:test 1} and \ref{fig:test 2} show that by increasing the number of grid points, the absolute error is reduced  which confirms the results obtained in Section \ref{thistable}. The same conclusion is true in Figure \ref{fig:test 2} which shows the computed results for the method FH($d$, $m$). Furthermore, Figure \ref{modifiedcom} gives a clear evidence  for choosing a new strategy in the numerical solution of nonlinear system of equations presented in (\ref{dis}). In fact it could be seen that there is a meaningful difference in computing times when we use the ``Newton-interpolation'' scheme.

In summary, we conclude this section by noting that if the speed of computation is the main criteria in choosing a specific pricing framework, we could use the BFH$(d,m)$ method which also provides an acceptable error both in the free boundary and also the price. 
\begin{figure}
\centering
\begin{subfigure}{.5\textwidth}
  \centering
  \includegraphics[width=1.08\textwidth, height=0.2\textheight]{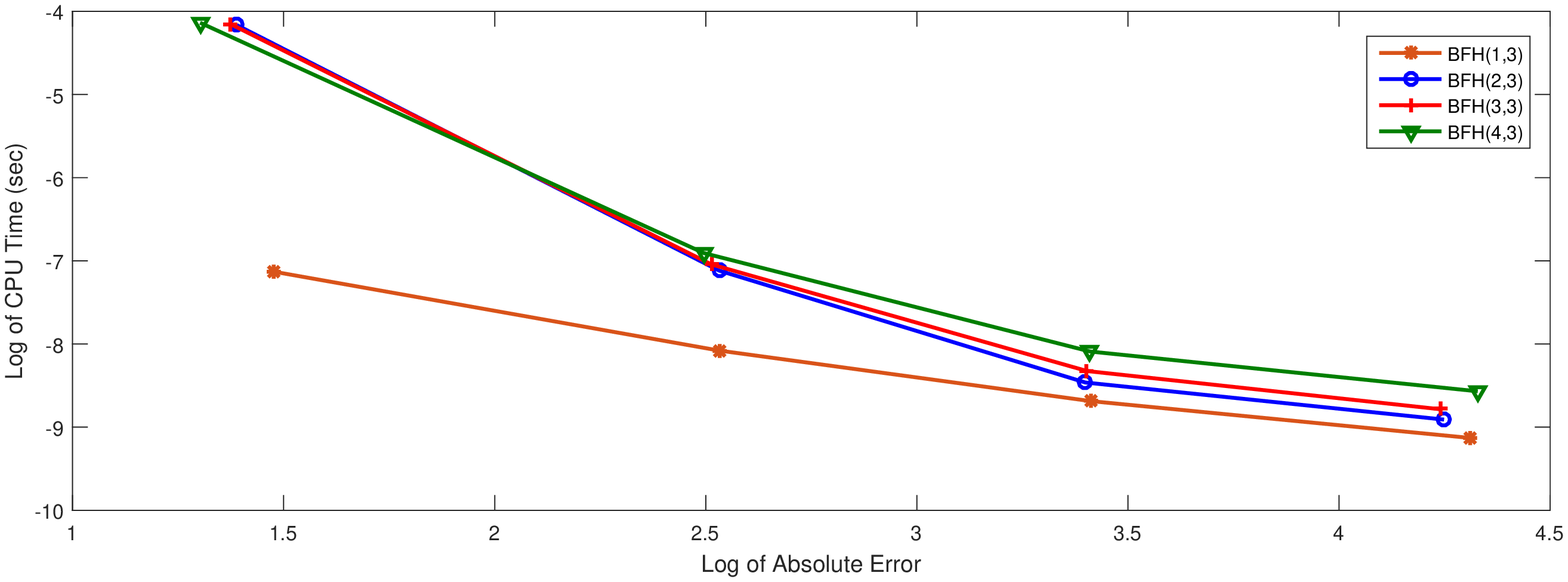}
  \label{fig:sub51}
\end{subfigure}%
\begin{subfigure}{.5\textwidth}
  \centering
  \includegraphics[width=1.08\textwidth, height=0.2\textheight]{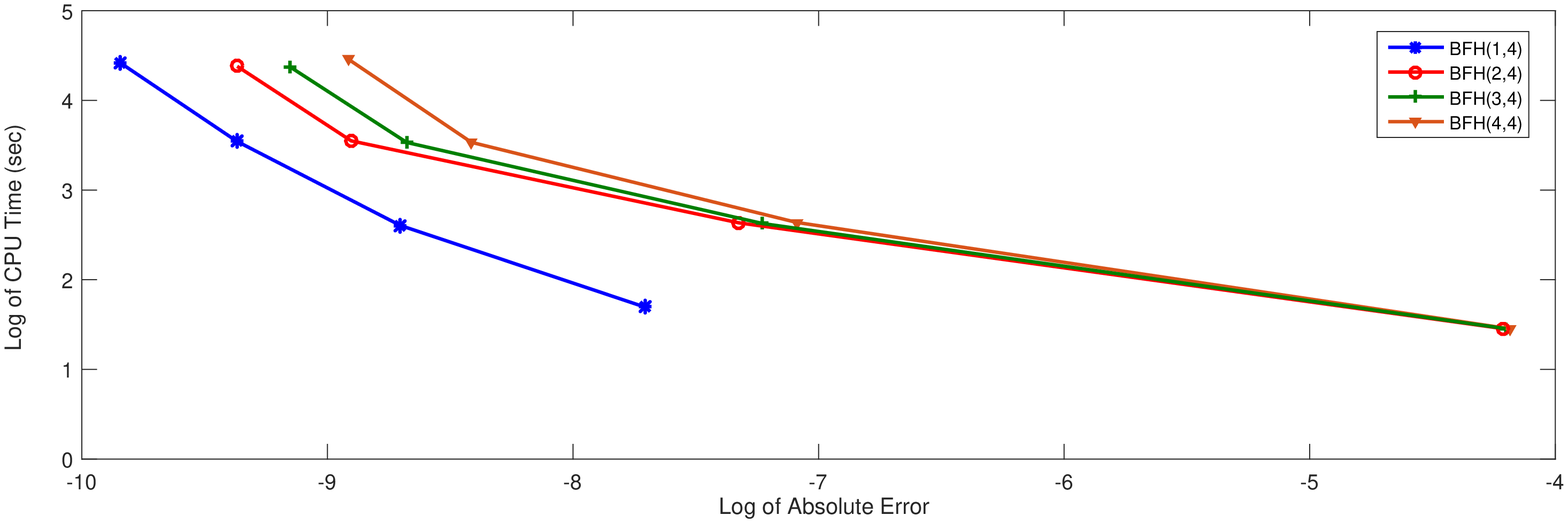}
  \label{fig:sub52}
\end{subfigure}
\caption{Work precision diagrams for Berrut and Floater-Hormann method }
\label{fig:test 1}
\end{figure}
 \begin{figure}
\centering
\begin{subfigure}{.5\textwidth}
  \centering
  \includegraphics[width=1.07\textwidth, height=0.2\textheight]{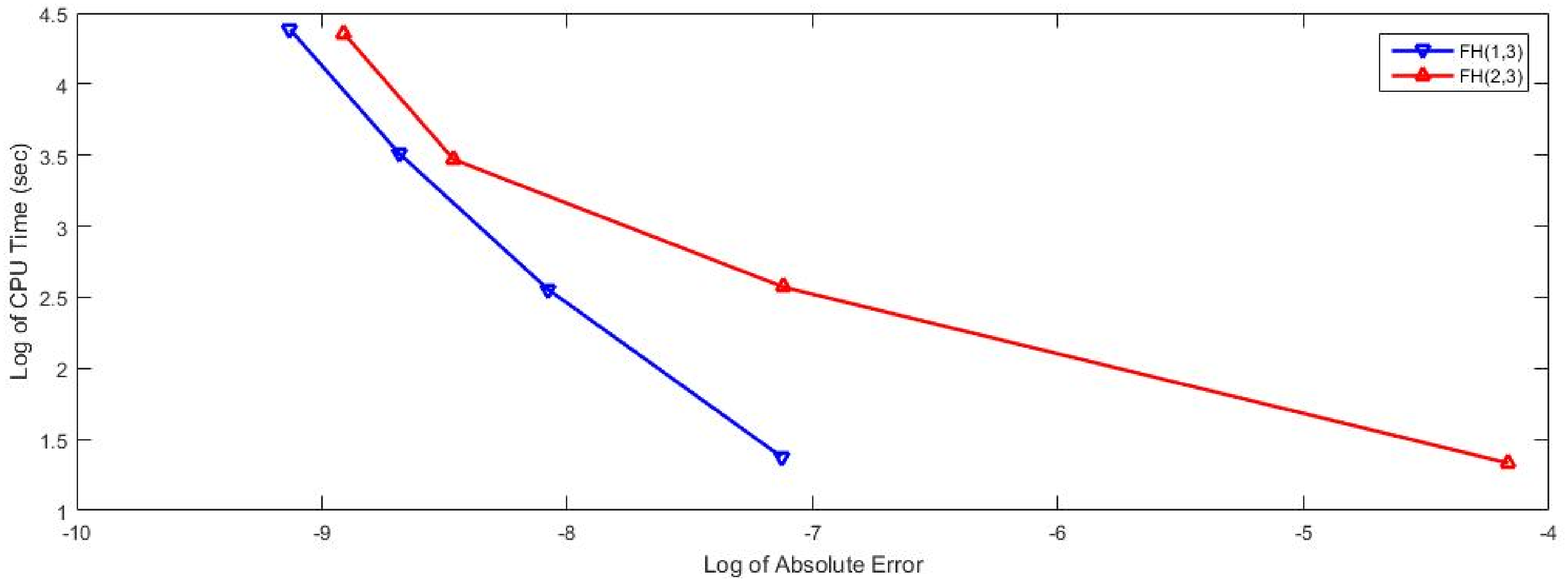}
  \caption{}
  \label{fig:sub31}
\end{subfigure}%
\begin{subfigure}{.5\textwidth}
  \centering
  \includegraphics[width=1.07\textwidth, height=0.2\textheight]{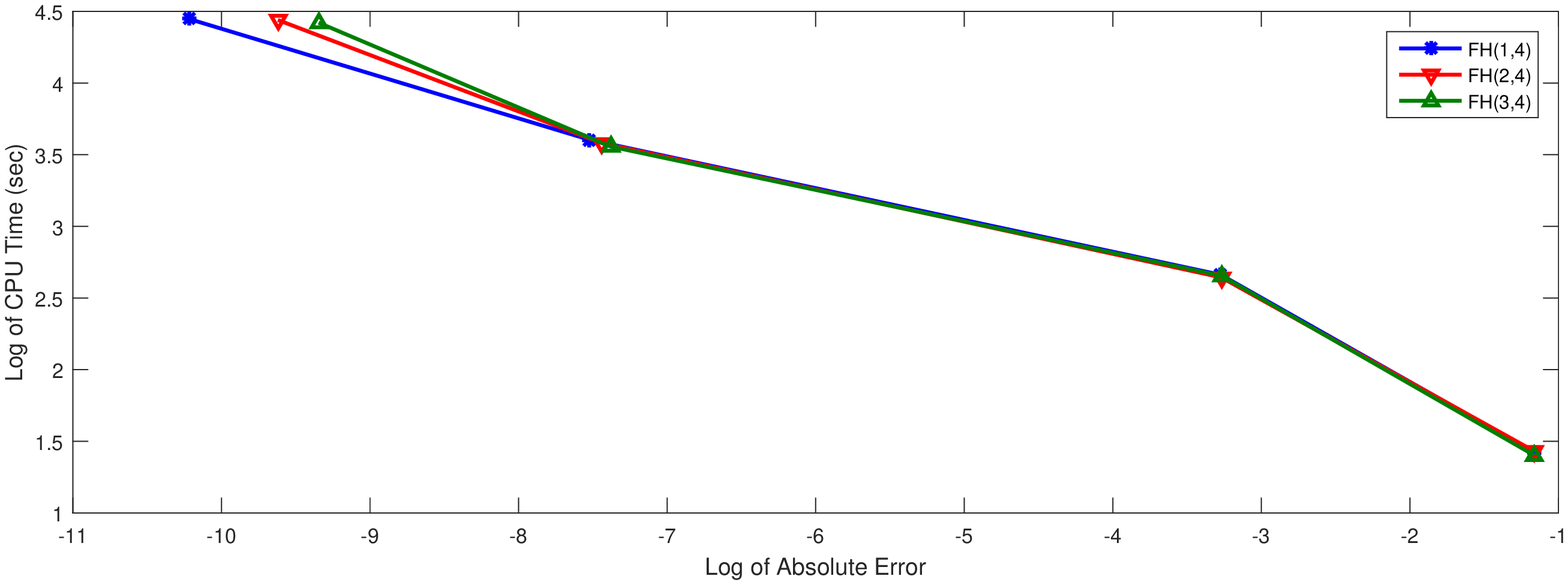}
  \caption{}
  \label{fig:sub32}
\end{subfigure}
\caption{Work precision diagrams for Floater-Hormann method }
\label{fig:test 2}
\end{figure}
\begin{figure}[h]
	\begin{center}
		\includegraphics[width=10.01cm, height=5cm]{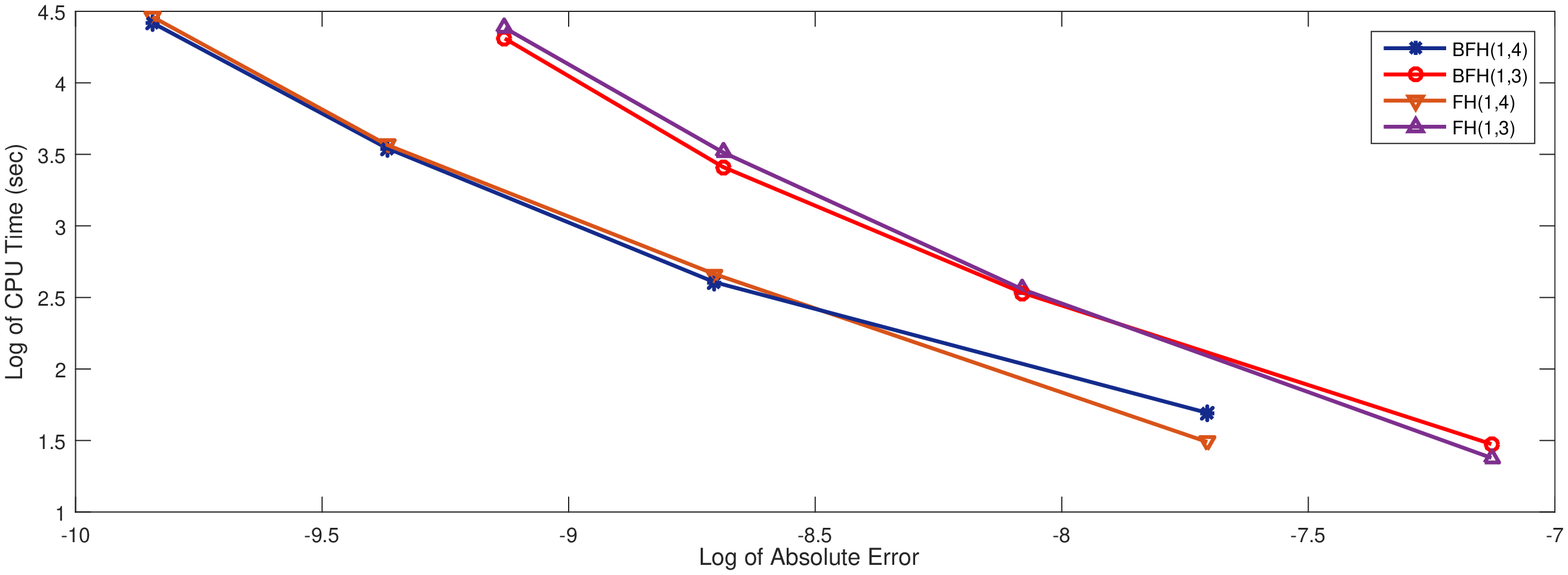}
		\caption{Comparison of Berrut and Floater-Hormann method with Floater-Hormann method}
		\label{test222}
	\end{center}
\end{figure}
\begin{figure}[h]	
	\begin{center}
		\includegraphics[width=10.01cm, height=5cm]{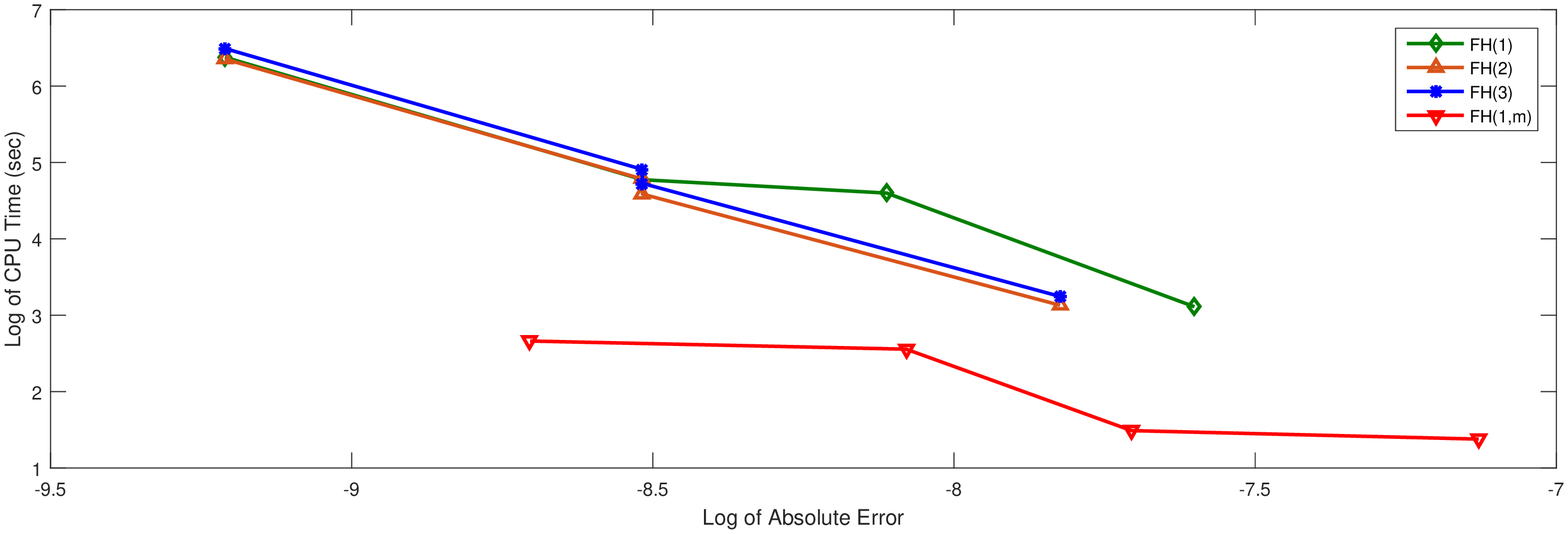}
		\caption{Comparison of Newton-Interpolation method with Floater-Hormann method }
		\label{modifiedcom}
	\end{center}                                            
\end{figure}
\section{Conclusion and Further Remarks}\label{conclusion}
In this paper, some integral equation representations describing the early exercise boundary of an American option were considered. We also reviewed some numerical approaches employed in the current literature to solve for the early exercise boundary based on these integral equation classes. 
The existence and uniqueness issue is discussed for some classes of these integral equations which could be extended to other classifications. We also discussed the problem of equivalence between these integral equation representations. By employing a revised form of Kim's integral representation of the free boundary and because of the weakly singular behavior of the kernel, a product integration method based on the barycentric rational 
quadrature is proposed to compute the American put price. We also have provided a theoretical analysis of the proposed method as well as some 
numerical evidence concerning the accuracy and efficiency of this framework. This work could be extended by studying the numerical stability as well as extending this framework to the numerical study of other integral equation classes, specially those leading to Urysohn type first kind integral equations defined on an unbounded domain. Extension to integral equations arising from more complicated dynamics such as jump-diffusions will be also worthy of investigation.
\begin{appendices}
\section{Fourier Transform Approach}\label{forieh}
In the literature, Fourier transform is used in complete  and incomplete forms to reformulate  the option pricing problem.
\subsection{McKean's Approach}
 Let us define the Fourier and incomplete Fourier  transforms of $V(t, S)$ as
\begin{equation}\mathcal{F}\{V(t, S) \} = \int_{-\infty}^{\infty}e^{\im\omega S} V(t, S) \mathrm{d}S, \quad \mathcal{F}_{b}\{V(t, S)\} = \int_{b}^{\infty}e^{\im\omega S} V(t, S) \mathrm{d}S, \end{equation}
for $b<S<\infty$.  Chiarella et al. \cite{chiarella2014numerical, chia} inspired by McKean's work  \cite{kean} derived a fully nonlinear Volterra integro-differential equation by applying  the change of variable $S = e^{x}$, as well as the incomplete Fourier transform to Eq. (\ref{pde}) as follows
\begin{equation}\label{aa1}
\begin{split}
\frac{v(\ln \mathcal{B}(t))}{2}  = & \frac{e^{-rt}}{\sigma \sqrt{2 \pi t}}\int_{-\infty}^{\ln \mathcal{B}(0^{+})} e^{-\frac{(\ln \mathcal{B}(t) -u-k\tau)^{2}}{2 \sigma^{2}\tau}} v(u)\mathrm{d}u \\
 & +  \int_{0}^{t} \frac{e^{-r(t-s)}}{\sigma \sqrt{ 2 \pi (t-s)}} \left[ e^{-h(\ln \mathcal{B}(t), t, s)} Q(\ln \mathcal{B}(t),t, s) \right] \mathrm{d}s.
\end{split}
\end{equation}
In the above formula, we have used the notations
 \[ v(x)\equiv \max \{ e^{x} -K, 0 \}, \]
\[ h(x, t, s) = \frac{(x - \ln \mathcal{B}(s) + k(t-s))^{2}}{2 \sigma^{2}(t-s)},\]
\[ Q(x,t,  s) = \frac{\sigma^{2} v' (\ln \mathcal{B}(s))}{2} + \left(
\frac{\mathcal{B}'(s)}{\mathcal{B}(s)} + \frac{1}{2} \left[ k - \frac{(x- \ln \mathcal{B}(s))}{(t-s)} \right]  v(\ln (\mathcal{B}(s) \right),
\]
and
\[k = r - \delta - \frac{1}{2} \sigma^{2}.\]
\subsection{Chadam-Stamicar-\v{S}ev\v{c}ovi\v{c}'s Approach}
\v{S}ev\v{c}ovi\v{c} \cite{shev2} and Stamicar et al. \cite{stamicar}  have utilized  the Fourier sine and cosine transforms defined as
\begin{align*}
\mathcal{F}_{s}\{ V(t, S)\} = &\int_{0}^{\infty}V(t, S)\sin(\omega S)\mathrm{d}S,\\
\mathcal{F}_{c}\{ V(t, S)\} = &\int_{0}^{\infty}V(t, S)\cos(\omega S)\mathrm{d}S,\end{align*}
to find the
American  option price. In the zero-dividend case, they proved that  the early exercise boundary of an American put satisfies the integral equation  defined recursively as follows
\begin{equation}\label{Aa2}
\begin{split}
  \eta (t) = & - \sqrt{- \ln \left[ \sqrt{\pi} \sqrt{t} \exp \left( \frac{2r}{\sigma^{2}}\right) \left(1- \frac{F(t)}{\sqrt{\pi}} \right)   \right]},\\
 g(t, \theta) = & \frac{1}{\cos \theta} \left[ \eta (t) -\sin \theta \eta (t \sin^{2}\theta) \right],\\
 F(t)  = & 2\int_{0}^{\frac{\pi}{2}} \exp \left( -  \frac{2r}{\sigma^{2}} t \cos^{2} \theta - g^{2}(t, \theta) \right) \Big\{ \sqrt{t} \sin \theta + g(t, \theta)\tan \theta\Big\} \mathrm{d}\theta,
\end{split}
\end{equation}
and the early exercise boundary is obtained by the formula  \[\mathcal{B}(t) = K \exp \left( - \left( \frac{2r}{\sigma^{2}}-1\right) t \right) \exp \left( 2 \sqrt{t} \eta (t)\right).\]
By the change of variable $s= t \sin^{2} \theta $, the following fully nonlinear weakly singular Volterra integral equation is obtained
\begin{equation} \label{shevin}F (t) =  \int_{0}^{t}\exp \left( -  \frac{2r}{\sigma^{2}} (t-s) - \frac{(\sqrt{t}\eta (t) - \sqrt{s}\eta (s))^{2}}{t-s} \right)  \Big\{ 1+ \frac{\sqrt{t}\eta (t) - \sqrt{s}\eta (s)}{t-s} \Big\} \frac{\mathrm{d}s}{\sqrt{t-s}}.\end{equation}
For the divided paying case,  they have extended this approach (written here for a call option) have shown that the early exercise boundary satisfies
\begin{equation}\label{shevon}
\begin{split}
\mathcal{B}(t) = & \frac{r K}{\delta} \Big( 1+ \frac{\sigma}{r \sqrt{2\pi t}}
\exp \Big(  -rt - \frac{(A(t,s)+ \ln(\frac{r}{\delta}))^{2}}{2 \sigma^{2}t} \Big)\Big)  \\
& + \frac{1}{\sqrt{2 \pi}}\int_{0}^{t} \Big[  \sigma + \frac{1}{\sigma}
(1 - \frac{\delta \mathcal{B}(s)}{r K})\frac{A(t,s)}{t-s} \Big] \frac{\exp \Big( -r (t-s) - \frac{A(t,s)^{2}}{2 \sigma^{2}(t-s)}\Big)}{\sqrt{t-s}}\mathrm{d}s,
\end{split}
\end{equation}
where the function $A$ is defined as
\[A(t,s) = \ln \frac{\mathcal{B}(t)}{\mathcal{B}(s)} +  \left(   r- \delta - \frac{\sigma^{2}}{2}\right) \left( t-s\right). \]
\section{Laplace Transform Approach}\label{laplas}
 As it is usual in the literature of partial differential equation, this transformation could be used to reduce the dimension of equation. This idea is used by some researchers in order to find an appropriate solution for the free boundary problem by reducing it to an integral equation which is reviewed in the following.
\subsection{Knessl's Approach} Knessl \cite{knessl} use the idea of ``moving reference frame" to convert the free boundary problem (\ref{pde})-(\ref{tah}) to a fixed boundary value problem.
By introducing new variables, he converts Eq. (\ref{pde}) into a PDE with constant coefficients
\begin{align}
p_{t}&=p_{xx}+(\rho -1)p_{x}, \quad x>b(t), \quad t>0,\\
b(0)&=0,\\
p(0,x)&=e^{x}-1, \quad x\geq 0,\\
p(t,b(t))&=e^{\rho t}-1, \quad p_{x}(t,b(t))=0, \quad t>0.
\end{align}
Then by a new variable $y=x-b(t)$, free boundary problem is converted to a fixed boundary value problem given as
\begin{align}
p_{t}&=p_{xx}+\Big[\rho-1+b'(t)\Big]p_{y}, \quad y>0, \quad t>0,\\
p(0,y)&=e^{y}-1, \quad y\geq0,\\
p(t,0)&=e^{\rho t}-1, \quad p_{y}(t,0)=0, \quad t>0.
\end{align}
Applying the Laplace transform
\[\mathcal{L}\{p(t,x)\} = \int_{0}^{\infty} p(t,y)e^{-sy}\mathrm{d}y,\]
to this PDE leads to the following nonlinear integral equation for $b(t)$ as
\begin{equation}\label{bb1}
\frac{1}{s-1} = \frac{2r}{\sigma^{2}}\int_{0}^{\infty} \exp \Big(\frac{2r}{\sigma^{2}} t -s {b}(t) -s (s + \frac{2r}{\sigma^{2}} -1) \Big) \mathrm{d}t, \quad \Re (s) > 1,  \end{equation}
and finally the early exercise boundary is obtained as $\mathcal{B}(t)=K e^{b(t)}$. It is seen that the above equation  is a Fredholm integral equation of the first kind.
\subsection{Mallier-Alobaidi's Approach}
Laplace transform in time is used to Eq. (\ref{pde}) with the conditions (\ref{smoothput1})-(\ref{tah}) and also in order to tackle the difficulty of holding the Black-Scholes-Merton PDE, they utilize incomplete Laplace transform and obtain an integral equation for the early exercise boundary.
To introduce this approach, we define the notations
\[
S_{0} = \frac{Kr}{\delta}, \quad \alpha^{+}= \frac{1}{2\sigma^{2}}\left[  \sigma^{2} -2(r-\delta) + \sqrt{4 \delta^{2} - 8 \delta r + 4 \delta \sigma^{2} + 4 r^{2} + 4\sigma^{2}r + \sigma^{4} } \right].
 \]
Let $S^{*} =  \frac{K}{1-\frac{1}{\alpha^{+}}}$. It can be shown that for $r>\delta>0$, the early exercise boundary of the American call satisfies  the following equation
 \begin{equation}\label{l1}
  \begin{split}
 \int_{S_{0}}^{S^{*}} S^{\frac{-1}{2 \sigma^{2}}\left( 2 \delta -2r + 3\sigma^{2}-\lambda(p)\right) } F(S)\mathrm{d}S&  = \frac{1}{4}e^{pT}K^{\frac{-1}{2\sigma^{2}}\left( 2 \delta -2r -3 \sigma^{2}-\lambda(p) \right) }\\
 & \times \left[ 1 - (\frac{r}{\delta})^{\frac{-1}{2\sigma^{2}}\left( 2 \delta -2r - \sigma^{2}-\lambda(p) \right) } \right] \\
 & \times \left[ \frac{2\delta -2r - \sigma^{2}+ \lambda(p)}{p+ \delta} -  \frac{2\delta -2r + \sigma^{2}+ \lambda(p)}{p+ r} \right],
\end{split}
\end{equation}
where
\[  \lambda (p) = \sqrt{ 4 \delta^{2} - 8 \delta r + 4 \delta \sigma^{2} + 4 r^{2} + 4\sigma^{2}r + \sigma^{4} + 8\sigma^{2} p},\]
and
\[ F(S) = (S - K)e^{p T_{f}(S)} - \left[ (r-\delta)(K-S)S - \sigma^{2}S^{2} \right] T'_{f}(S) -\frac{1}{2}\sigma^{2}S^{2}(S-K)T''_{f}(S).  \]
In the above equation, $T_{f}(S)$ is the early exercise boundary in the Laplace space \cite{gada1}. Furthermore, it can be proved that the early exercise boundary of the American put solves the equation
\begin{equation}\label{l2}
 \int_{S^{*}}^{K} S^{-\frac{1}{2\sigma^{2}}\left[ 2\delta -2r +3 \sigma^{2} +  \lambda(p)\right]}  F(S) \mathrm{d}S = 0. \end{equation}
Both of equations, (\ref{l1}) and (\ref{l2}) could be categorized as the Urysohn integral equations of the first kind.
\section{Mellin Transform Approach}\label{melina}
\subsection{Mellin Transform}The Mellin transform of $V(t,S)$ defined by
\[\mathcal{M}\{V(t, S)\} = \int_{0}^{\infty} V(t, S)S^{\omega -1}\mathrm{d}S,\]
is applied to Eq. (\ref{pde}) with the conditions (\ref{putcon})-(\ref{tah}) to obtain the following inhomogeneous ordinary differential equation
\[
\frac{d\widehat{P}}{dt} + \Big( \frac{\sigma^2}{2}(\omega^{2}+\omega)-r\omega -r \Big)\widehat{P} = \frac{-rK}{\omega}(\mathcal{B}(t))^{\omega}.
\]
Solving this ODE gives
\begin{align*}
\widehat{P}(t, \omega) =A(\omega)e^{-\frac{1}{2}\sigma^{2}q(\omega)t} +\frac{rK}{\omega}\int_{t}^{T}(\mathcal{B}(s))^{\omega}e^{\frac{1}{2}\sigma^{2}q(\omega)(s-t)}\mathrm{d}s,
\end{align*}
where $Q(\omega) = \omega^{2} + \omega  \left( 1- \frac{2(r-\delta)}{\sigma^{2}}\right) - \frac{2r}{\delta}$.
Finally using the inversion of the Mellin transform, we arrive at the following representation for the  put price
\begin{align*}
P(t,S) =& \frac{1}{2\pi \im}\int_{c-\im \infty}^{c+\im \infty} \widehat{\theta}(\omega) e^{\frac{1}{2}\sigma^{2} Q(\omega)(T-t)} S^{-\omega}\mathrm{d}\omega \\
&+ \frac{rK}{2\pi \im}\int_{c-\im \infty}^{c+\im\infty} S^{-\omega} \int_{t}^{T}\frac{(\mathcal{B}(s))^{\omega}}{\omega}e^{\frac{1}{2}\sigma^{2}Q(\omega)(s-t)}\mathrm{d}s\mathrm{d}\omega.
\end{align*}
The above approach has been studied in \cite{frontczak2008pricing, panini} and it gives the following fully nonlinear Volterra integral equation for the early exercise boundary
\begin{equation}\label{melin}
\begin{split}
\mathcal{B}(t) - K  = p(t,\mathcal{B}(t)) & + \frac{1}{2\pi \im}\int_{c- \im \infty}^{c+\im \infty}\int_{t}^{T} \frac{r K}{\omega} \left(  \frac{\mathcal{B}(t)}{\mathcal{B}(s)} \right) ^{-\omega} e^{\frac{1}{2} \sigma^{2}Q(\omega)(s-t)} \mathrm{d}s \mathrm{d}\omega\\
& -  \frac{1}{2\pi \im}\int_{c-\im \infty}^{c+\im \infty}\int_{t}^{T} \frac{r \mathcal{B}(t)}{\omega +1} \left(  \frac{\mathcal{B}(t)}{\mathcal{B}(s)} \right) ^{-\omega} e^{\frac{1}{2} \sigma^{2}Q(\omega)(s-t)} \mathrm{d}s \mathrm{d}\omega.
\end{split}
\end{equation}
It could be shown that Eq. (\ref{melin}) is equivalent to Eq. (\ref{kim}) via the convolution property of the Mellin transform (for more details see \cite{frontczak2008pricing}).
\subsection{Modified Mellin Transform} Let $C^{E}(t,S)$ denote the European call option price. Since $C^{E}(t,S)= \mathcal{O}(1)$ for $S\rightarrow 0^{+}$ and $C^{E}(t,S)= \mathcal{O}(S)$ as $S\rightarrow\infty$, Frontczak and Sch\"{o}bel \cite{patrik2}  proposed  a modified Mellin transform defined by
\[\mathcal{M}(C^{E}(t,S), -\omega):=\int_{0}^{\infty}C^{E}(t,S) S^{-(\omega +1)}\mathrm{d}S. \]
They have shown that the price of an European call option is given by
\[C^{E}(t,S)=\frac{1}{2\pi \im}\int_{c-\im \infty}^{c+ \im \infty}K^{-\omega +1}\Big(\frac{1}{\omega-1}-\frac{1}{\omega}\Big)e^{\frac{1}{2}\sigma^{2}Q(\omega)(T-s)}S^{\omega}\mathrm{d}\omega,\]
which is equivalent to the Black-Scholes-Merton formula for European call price. They also showed that the price of an American call could be obtained by
\begin{align}
C^{A}(t,S) =& C^{E}(t,S)+ \frac{1}{2\pi \im}\int_{c-\im \infty}^{c+\im \infty}\int_{t}^{T} \frac{\delta \mathcal{B}(s)}{\omega -1}\Big(\frac{S}{\mathcal{B}(s)}\Big)^{\omega}e^{\frac{1}{2}\sigma^{2}Q(\omega)(s-t)}\mathrm{d}s\mathrm{d}\omega\\
&-\frac{1}{2\pi \im}\int_{c-\im \infty}^{c+\im \infty}\int_{t}^{T} \frac{rK}{\omega}\Big( \frac{S}{\mathcal{B}(s)}\Big)^{\omega} e^{\frac{1}{2}\sigma^{2}Q(\omega)(s-t)}\mathrm{d}s\mathrm{d}\omega,\notag
\end{align}
which is equivalence to Eq. (\ref{price}).
\section{Green's Function Approach}\label{green1}
\subsection{Zero-Dividend Case}
We consider Eq. (\ref{pde}) with the initial and boundary conditions (\ref{putcon}) and rewrite the dimensionless form of it with the variables:
\[ \rho = \frac{2r}{\sigma^{2}}, \quad S=Ke^{x}, \quad t=T-\frac{2}{\sigma^{2}}\tau, \quad b(t) = \log\Big[\frac{\mathcal{B}(t)}{K}\Big], \quad P(t,S)= Kp(x,\tau).\]
The fundamental solution of the reformulated PDE which is given by the Green's identity is as follows
\begin{equation}
p(\tau, x) = \int_{-\infty}^{0}(1-e^{y})\Gamma(x-y, \tau)\mathrm{d}y + \rho \int_{0}^{\tau}\int_{-\infty}^{b(\tau - s)}\Gamma(x-y, s)\mathrm{d}y\mathrm{d}s,
\end{equation}
where
\[\Gamma (\tau, x) = \frac{1}{\sqrt{4 \pi \tau}}e^{-\frac{[x+(\rho - 1)\tau]^{2}}{4 \tau}- \rho \tau}. \]
The above expression for the price solves Eq. (\ref{pde}) as well as the early exercise boundary satisfies the following integral and integro-differential  equations
\begin{equation}
\begin{split}
\int_{0}^{\tau}\Gamma (s, b(\tau))\mathrm{d}s =& \, \rho \int_{0}^{\tau}\int_{b(\tau -s)}^{0} \Gamma (s, b(\tau)-y)\mathrm{d}y\mathrm{d}s,\\
\int_{0}^{\tau} \Gamma_{x}(s, b(\tau)) + \rho \Gamma (s, b(\tau))\mathrm{d}s =& \,\rho \int_{0}^{\tau} \Gamma (s, b(\tau) - b(\tau -s))\mathrm{d}s,\\
\Gamma ( \tau, b(\tau)) = &\, - \rho \int_{0}^{\tau} \Gamma(s,b(\tau) - b(\tau -s)) b'(\tau -s )\mathrm{d}s,\\
\Gamma ( \tau, b(\tau)) = & \, \frac{\rho}{2} + \rho \int_{0}^{\tau} \Gamma_{x}(s,b(\tau) - b(\tau -s)) - \Gamma (s,b(\tau) - b(\tau -s)) \mathrm{d}s,\\
b'(\tau) = &\, - \frac{2 \Gamma_{x}( \tau, b(\tau))}{\rho} - 2 \int_{0}^{\tau} \Gamma_{x} (s,b(\tau) - b(\tau -s)) b'(\tau -s)\mathrm{d}s.
\end{split}
\end{equation}
\subsection{Non-Zero Dividend Case}
The above discussion can be extended for non-zero divided case. Let us introduce the Green's function for (\ref{pde})
\[ G(x, \tau; \xi,s) = \Gamma(x-\xi, \tau-s)e^{\rho(\tau-s)}.  \]
 For the case of $\delta >0 $, the price of the American put option is given by
 \begin{eqnarray}
  P(\tau, x) & = & \int_{0}^{\infty} (e^{\xi} - 1)\Gamma(x-\xi, \tau) e^{\rho \tau} \mathrm{d}\xi  \\\notag
& & + \int_{0}^{\tau}\int_{b(s)}^{\infty} ( e^{\xi} - \rho) e^{\rho (\tau-s)} \Gamma(x-\xi, \tau-s) \mathrm{d}\xi \mathrm{d}s,  \\\notag
& =&  I^{(1)}(\tau, x) + I^{(2)}(\tau, x)
\end{eqnarray}
  for more detail see \cite{evans}. The payoff condition implies $P_{\tau}( \tau, b(\tau)) = 0$, which gives
  \begin{equation}\label{gens}
   \frac{\partial I^{(1)}}{\partial \tau} [\tau, b(\tau)] = - \lim_{x\rightarrow b(\tau)} \frac{\partial I^{(2)}}{\partial \tau} [\tau,x].
\end{equation}
Eq. (\ref{gens}) equation is a weakly singular Volterra integral equation of nonlinear type.
The early exercise boundary introduced at (\ref{pde}) can be obtained by $\mathcal{B}(t) = K e^{b(t)}$.
\section{Optimal Stopping Approach} \label{stop}
It is a well-known fact that in a complete market and using arbitrage arguments, we could use the existence of a unique equivalent martingale measure, $Q$ to derive a unique price for both European and American option contracts. Among the early contributions to this field of research, one could mention   \cite{ben, karat, myneni} in which the authors show that the price of an  American put  option could be represented as the expected supremum of the discounted payoff function over all admissible stopping times, $\tau$,  of the form
\begin{equation}\label{optimal}
V(t,x) = \sup_{0\leq \tau \leq T-t} E_{t,x}\left(  e^{-r\tau}\left(  K- X_{t+\tau}\right)^{+}  \right),
\end{equation}
where $E_{t,x}[\cdot] = E_{Q}[\cdot|X_{t} = x]$.
 In (\ref{optimal}), the stochastic process  $X=(X_{t+s})_{s\geq 0}$ satisfies the geometric Brownian motion differential equation of the form
\[ d X_{t+s} = r X_{t+s} ds + \sigma X_{t+s}d B_{s}, \quad X_{t}= x,\]
with the exact solution
\[ X_{t+s} = x \exp \left( \sigma B_{s} + (r-\frac{\sigma^{2}}{2})s \right),  \]
in which $B = (B_{s})_{s\geq 0}$ denotes the standard Brownian motion process starting at zero and $ (x, t) \in (0, T]\times \Bbb{R} $  is given beforehand.
Under some regularity conditions on $V$, applying It\^{o}'s formula to $e^{-rs}V(t+s, X_{t+s})$ and taking the $P_{t,x}$-expectation on both sides of the resulting identity, we  obtain ``the early exercise premium representation'' of the form
\begin{equation}
\begin{split}
V(t,x) = e^{-r(T-t)}E_{t,x}(G(X_{T})) + rK \int_{0}^{T-t}e^{-ru} P_{t,x}(X \leq \mathcal{B}(t+u))\mathrm{d}u,
\end{split}
\end{equation}
where $G(x) = (K-x)^{+}$ (for more details see \cite{peskir}).
Applying the accompanying conditions (\ref{putcon}), one obtains the integral equation
 \begin{dmath*}
 K-\mathcal{B}(t)=e^{-r(T-t)}\int_{0}^{K}\aleph\Big(\dfrac{1}{\sigma \sqrt{T-t}}(\log(\dfrac{K-s}{\mathcal{B}(t)})-(r-\frac{\sigma^{2}}{2})(T-t))\Big)\mathrm{d}s+rK\int_{0}^{T-t}e^{-rs}\aleph\Big(\dfrac{1}{\sigma \sqrt{s}}(\log(\dfrac{\mathcal{B}(t+s)}{\mathcal{B}(t)})-(r-\dfrac{\sigma^{2}}{2})s)\Big)\mathrm{d}s,
\end{dmath*}
in the non-dividend paying case. Furthermore, when the option pays dividends,  Kim \cite{kim} has employed the risk-neutral valuation framework of Cox and Ross \cite{cox} to  obtain the nonlinear integral equation (\ref{kim}).
\end{appendices}
\bibliographystyle{plain}
\bibliography{bibpro}

\begin{thebibliography}{10}

\bibitem{achdou}
Y.~Achdou and O.~Pironneau.
\newblock {\em Computational Methods for Option Pricing}, volume~30 of {\em
  Frontiers in Applied Mathematics}.
\newblock Society for Industrial and Applied Mathematics (SIAM), Philadelphia,
  PA, 2005.

\bibitem{lai2}
F.~Aitsahlia and T.~L. Lai.
\newblock Exercise boundaries and efficient approximations to {A}merican option
  prices and hedge parameters.
\newblock {\em Journal of Computational Finance}, 4(4):85--104, 2001.

\bibitem{cuminato}
S.~S. Allaei, T.~Diogo, and M.~Rebelo.
\newblock Analytical and computational methods for a class of nonlinear
  singular integral equations.
\newblock {\em Applied Numerical Mathematics}, 114:2--17, 2017.

\bibitem{alobaidi2014integral}
G.~Alobaidi, R.~Mallier, and M.C. Haslam.
\newblock Integral transforms and {A}merican options: {L}aplace and {M}ellin go
  {G}reen.
\newblock {\em Acta Mathematica Universitatis Comenianae}, 83(2):245--266,
  2014.

\bibitem{andersen}
L.~B.~G Andersen, M.~Lake, and D.~Offengenden.
\newblock High-performance {A}merican option pricing.
\newblock {\em Journal of Computational Finance}, 20(1):39--87, 2016.

\bibitem{atkinson}
K.~Atkinson and H.~Weimin.
\newblock {\em Theoretical Numerical Analysis: A Functional Analysis
  Framework}, volume~39 of {\em Texts in Applied Mathematics}.
\newblock Springer, New York, second edition, 2005.

\bibitem{orsi}
P.~Baratella and A.~P. Orsi.
\newblock A new approach to the numerical solution of weakly singular
  {V}olterra integral equations.
\newblock {\em Journal of Computational and Applied Mathematics},
  163(2):401--418, 2004.

\bibitem{barron}
G.~Barone-Adesi and R.~E. Whaley.
\newblock Efficient analytic approximation of {A}merican option values.
\newblock {\em The Journal of Finance}, 42(2):301--320, 1987.

\bibitem{ben}
A.~Bensoussan.
\newblock On the theory of option pricing.
\newblock {\em Acta Applicandae Mathematicae}, 2(2):139--158, 1984.

\bibitem{bos}
L.~Bos, S.~De~Marchi, K.~Hormann, and G.~Klein.
\newblock On the {L}ebesgue constant of barycentric rational interpolation at
  equidistant nodes.
\newblock {\em Numerische Mathematik}, 121(3):461--471, 2012.

\bibitem{broadie1996american}
M.~Broadie and J.~Detemple.
\newblock {A}merican option valuation: new bounds, approximations, and a
  comparison of existing methods.
\newblock {\em Review of Financial Studies}, 9(4):1211--1250, 1996.

\bibitem{brunner}
H.~Brunner.
\newblock {\em Collocation Methods for {V}olterra Integral and Related
  Functional Differential Equations}, volume~15 of {\em Cambridge Monographs on
  Applied and Computational Mathematics}.
\newblock Cambridge University Press, Cambridge, 2004.

\bibitem{bunch1992simple}
D.~S. Bunch and H.~Johnson.
\newblock A simple and numerically efficient valuation method for {A}merican
  puts using a modified {G}eske-{J}ohnson approach.
\newblock {\em The Journal of Finance}, 47(2):809--816, 1992.

\bibitem{carr1}
P.~Carr, R.~Jarrow, and R.~Myneni.
\newblock Alternative characterizations of {A}merican put options.
\newblock {\em Mathematical Finance}, 2(2):87--106, 1992.

\bibitem{chen2012new}
F.~Chen, J.~Shen, and H.~Yu.
\newblock A new spectral element method for pricing {E}uropean options under
  the {B}lack--{S}choles and {M}erton jump diffusion models.
\newblock {\em Journal of Scientific Computing}, 52(3):499--518, 2012.

\bibitem{chen}
X.~Chen and J.~Chadam.
\newblock A mathematical analysis of the optimal exercise boundary for
  {A}merican put options.
\newblock {\em SIAM Journal on Mathematical Analysis}, 38(5):1613--1641,
  2006/07.

\bibitem{chen1}
X.~Chen, H.~Cheng, and J.~Chadam.
\newblock New results for the {A}merican put option.
\newblock {\em Canadian Applied Mathematics Quarterly}, 17(4):615--626, 2009.

\bibitem{chiarella2014numerical}
C.~Chiarella, B.~Kang, and G.~H. Meyer.
\newblock {\em The Numerical Solution of the {A}merican Option Pricing Problem:
  Finite Difference and Transform Approaches}.
\newblock World Scientific, 2014.

\bibitem{chia}
C.~Chiarella, A.~Ziogas, and A.~Kucera.
\newblock A survey of the integral representation of {A}merican option prices.
\newblock Technical report, University of Technology Sydney, 2004.

\bibitem{corduneanu1991integral}
C.~Corduneanu.
\newblock {\em Integral Equations and Applications}.
\newblock Cambridge University Press, Cambridge, 1991.

\bibitem{cox}
J.~C. Cox and S.~A. Ross.
\newblock The valuation of options for alternative stochastic processes.
\newblock {\em Journal of Financial Economics}, 3(1-2):145--166, 1976.

\bibitem{hoog}
F.~de~Hoog and R.~Weiss.
\newblock High order methods for a class of {V}olterra integral equations with
  weakly singular kernels.
\newblock {\em SIAM Journal on Numerical Analysis}, 11:1166--1180, 1974.

\bibitem{MR2087015}
D.~G. Duffy.
\newblock {\em Transform Methods for Solving Partial Differential Equations}.
\newblock Chapman \& Hall/CRC, Boca Raton, FL, second edition, 2004.

\bibitem{duffy}
D.~J. Duffy.
\newblock {\em Finite Difference Methods in Financial Engineering: A Partial
  Differential Equation Approach}.
\newblock Wiley Finance Series. John Wiley \& Sons, Ltd., Chichester, 2006.

\bibitem{evans}
J.~D. Evans, R.~Kuske, and J.~B. Keller.
\newblock {A}merican options on assets with dividends near expiry.
\newblock {\em Mathematical Finance}, 12(3):219--237, 2002.

\bibitem{kai}
M.~S. Floater and K.~Hormann.
\newblock Barycentric rational interpolation with no poles and high rates of
  approximation.
\newblock {\em Numerische Mathematik}, 107(2):315--331, 2007.

\bibitem{frontczak2008pricing}
R.~Frontczak and R.~Sch{\"o}bel.
\newblock Pricing {A}merican options with {M}ellin transforms.
\newblock Technical report, T{\"u}binger Diskussionsbeitrag, 2008.

\bibitem{patrik2}
R.~Frontczak and R.~Sch\"{o}bel.
\newblock On modified {M}ellin transforms, {G}auss-{L}aguerre quadrature, and
  the valuation of {A}merican call options.
\newblock {\em Journal of Computational and Applied Mathematics},
  234(5):1559--1571, 2010.

\bibitem{geske1984american}
R.~Geske and H.~E. Johnson.
\newblock The {A}merican put option valued analytically.
\newblock {\em The Journal of Finance}, 39(5):1511--1524, 1984.

\bibitem{vellek1}
O.~E. G\"ottsche and M.~H. Vellekoop.
\newblock The early exercise premium for the {A}merican put under discrete
  dividends.
\newblock {\em Mathematical Finance}, 21(2):335--354, 2011.

\bibitem{guan}
Q.~Guan, R.~Zhang, and Y.~Zou.
\newblock Analysis of collocation solutions for nonstandard {V}olterra integral
  equations.
\newblock {\em IMA Journal of Numerical Analysis}, 32(4):1755--1785, 2012.

\bibitem{hack}
W.~Hackbusch.
\newblock {\em Integral Equations: Theory and Numerical Treatment}, volume 120
  of {\em International Series of Numerical Mathematics}.
\newblock Birkh\"{a}user Verlag, Basel, 1995.

\bibitem{hack2}
W.~Hackbusch.
\newblock {\em The Concept of Stability in Numerical Mathematics}, volume~45 of
  {\em Springer Series in Computational Mathematics}.
\newblock Springer, Heidelberg, 2014.

\bibitem{heider2007condition}
P.~Heider.
\newblock The condition of the integral representation of {A}merican options.
\newblock {\em Journal of Computational Finance}, 11(2):95, 2007.

\bibitem{heider}
P.~Heider.
\newblock A second-order {N}ystr\"om-type discretization for the early-exercise
  curve of {A}merican put options.
\newblock {\em International Journal of Computer Mathematics}, 86(6):982--991,
  2009.

\bibitem{kai112}
K.~Hormann.
\newblock Barycentric interpolation.
\newblock In {\em Approximation theory {XIV}: {S}an {A}ntonio 2013}, volume~83
  of {\em Springer Proc. Math. Stat.}, pages 197--218. Springer, 2014.

\bibitem{little}
C.~Hou, T.~Little, and V.~Pant.
\newblock A new integral representation of the early exercise boundary for
  {A}merican put options.
\newblock {\em Journal of Computational Finance}, 3(73-96), 2000.

\bibitem{huang1996pricing}
J.~Huang, M.~G. Subrahmanyam, and G.~G. Yu.
\newblock Pricing and hedging {A}merican options: a recursive integration
  method.
\newblock {\em Review of Financial Studies}, 9(1):277--300, 1996.

\bibitem{jacka}
S.~D. Jacka.
\newblock Optimal stopping and the {A}merican put.
\newblock {\em Mathematical Finance}, 1(2):1--14, 1991.

\bibitem{jam}
F.~Jamshidian.
\newblock An analysis of {A}merican options.
\newblock {\em Review of Futures Markets}, 11(1):72--80, 1992.

\bibitem{jerri1999introduction}
A.~J. Jerri.
\newblock {\em Introduction to Integral Equations with Applications}.
\newblock Wiley-Interscience, New York, second edition, 1999.

\bibitem{vellek2}
B.~Jourdain and M.~H. Vellekoop.
\newblock Regularity of the exercise boundary for {A}merican put options on
  assets with discrete dividends.
\newblock {\em SIAM Journal on Financial Mathematics}, 2(1):538--561, 2011.

\bibitem{ju}
N.~Ju.
\newblock Pricing an {A}merican option by approximating its early exercise
  boundary as a multipiece exponential function.
\newblock {\em Review of Financial Studies}, 11(3):627--646, 1998.

\bibitem{kallast}
S.~Kallast and A.~Kivinukk.
\newblock Pricing and hedging {A}merican options using approximations by {K}im
  integral equations.
\newblock {\em European Finance Review}, 7(3):361--383, 2003.

\bibitem{karat}
I.~Karatzas.
\newblock On the pricing of {A}merican options.
\newblock {\em Applied Mathematics and Optimization}, 17(1):37--60, 1988.

\bibitem{karatzas1998methods}
I.~Karatzas and S.~E. Shreve.
\newblock {\em Methods of Mathematical Finance}, volume~39.
\newblock Springer, 1998.

\bibitem{kim}
I.~J. Kim.
\newblock The analytic valuation of {A}merican options.
\newblock {\em Review of Financial Studies}, 3(4):547--572, 1990.

\bibitem{kim2}
I.~J. Kim, B.-G. Jang, and K.~T. Kim.
\newblock A simple iterative method for the valuation of {A}merican options.
\newblock {\em Quantitative Finance}, 13(6):885--895, 2013.

\bibitem{kleinthesis}
G.~Klein.
\newblock {\em Applications of Linear Barycentric Rational Interpolation}.
\newblock PhD thesis, University of Fribourg (Switzerland), 2012.

\bibitem{knessl}
C.~Knessl.
\newblock A note on a moving boundary problem arising in the {A}merican put
  option.
\newblock {\em Studies in Applied Mathematics}, 107(2):157--183, 2001.

\bibitem{keller}
R.~A. Kuske and J.~B. Keller.
\newblock Optimal exercise boundary for an {A}merican put option.
\newblock {\em Applied Mathematical Finance}, 5(2):107--116, 1998.

\bibitem{lauko}
M.~Lauko and D.~{\v{S}}ev{\v{c}}ovi{\v{c}}.
\newblock Comparison of numerical and analytical approximations of the early
  exercise boundary of {A}merican put options.
\newblock {\em The ANZIAM Journal}, 51(4):430--448, 2010.

\bibitem{gada1}
R.~Mallier and G.~Alobaidi.
\newblock Laplace transforms and {A}merican options.
\newblock {\em Applied Mathematical Finance}, 7(4):241--256, 2000.

\bibitem{kean}
H.~P. McKean.
\newblock Appendix: A free boundary value problem for the heat equation arising
  from a problem in mathematical economics.
\newblock {\em Industrial Management Review}, 1965.

\bibitem{cortazar}
L.~Medina.
\newblock {\em A Parallel Algorithm for Pricing {A}merican Options}.
\newblock PhD thesis, Pontificia Universidad Catolica De Chile, 2013.

\bibitem{myneni}
R.~Myneni.
\newblock The pricing of the {A}merican option.
\newblock {\em The Annals of Applied Probability}, pages 1--23, 1992.

\bibitem{nedaiasl2017numerical}
K.~Nedaiasl and A.~Foroush Bastani.
\newblock On the numerical approximation of some non-standard {V}olterra
  integral equations.
\newblock {\em Dolomites Research Notes on Approximation}, 10:118--127, 2017.

\bibitem{orsi1996product}
A.~Orsi.
\newblock Product integration for {V}olterra integral equations of the second
  kind with weakly singular kernels.
\newblock {\em Mathematics of Computation}, 65(215):1201--1212, 1996.

\bibitem{panini}
R.~Panini and R.P. Srivastav.
\newblock Option pricing with {M}ellin transnforms.
\newblock {\em Mathematical and Computer Modelling}, 40(1):43 -- 56, 2004.

\bibitem{peskir}
G.~Peskir.
\newblock On the {A}merican option problem.
\newblock {\em Mathematical Finance}, 15(1):169--181, 2005.

\bibitem{stas}
S.~A. Sauter and C.~Schwab.
\newblock {\em Boundary Element Methods}, volume~39 of {\em Springer Series in
  Computational Mathematics}.
\newblock Springer-Verlag, Berlin, 2011.

\bibitem{stack}
I.~Stakgold and M.~Holst.
\newblock {\em Green's Functions and Boundary Value Problems}.
\newblock Pure and Applied Mathematics (Hoboken). John Wiley \& Sons, Inc.,
  Hoboken, NJ, third edition, 2011.

\bibitem{stamicar}
R.~Stamicar, D.~\v{S}ev\v{c}ovi\v{c}, and J.~Chadam.
\newblock The early exercise boundary for the {A}merican put near expiry:
  numerical approximation.
\newblock {\em The Canadian Applied Mathematics Quarterly}, 7(4):427--444,
  1999.

\bibitem{sullivan}
M.~A. Sullivan.
\newblock Valuing {A}merican put options using {G}aussian quadrature.
\newblock {\em Review of Financial Studies}, 13(1):75--94, 2000.

\bibitem{bary}
L.~N. Trefethen.
\newblock {\em Approximation Theory and Approximation Practice}.
\newblock Society for Industrial and Applied Mathematics, Philadelphia, 2013.

\bibitem{underwood2002integral}
R.~Underwood and J.~Wang.
\newblock An integral representation and computation for the solution of
  {A}merican options.
\newblock {\em Nonlinear Analysis: Real world Applications}, 3(2):259--274,
  2002.

\bibitem{shev2}
D.~\v{S}ev\v{c}ovi\v{c}.
\newblock Analysis of the free boundary for the pricing of an {A}merican call
  option.
\newblock {\em European Journal of Applied Mathematics}, 12(1):25--37, 2001.

\end{thebibliography}
\end{document}